\documentclass{amsart}
\usepackage{amssymb,amsmath,amsthm,amscd}
\usepackage{mathtools}
\usepackage{esint}
\usepackage{mathrsfs}
\usepackage{mathabx}
\usepackage[T1]{fontenc}
\usepackage{hyperref}
\hypersetup{colorlinks,linkcolor={red},citecolor={blue}} 
\usepackage{enumerate}
\usepackage{tikz}

\newtheorem{theorem}{Theorem}
\newtheorem{lemma}[theorem]{Lemma}

\newtheorem{definition}[theorem]{Definition}

\newtheorem{remark}[theorem]{Remark}
\newtheorem{proposition}[theorem]{Proposition}

\numberwithin{theorem}{section}
\numberwithin{equation}{section}

\def\N {{\mathbb N}}
\def\Z {{\mathbb Z}}

\def\R {{\mathbb R}}
\def\C {{\mathbb C}}

\DeclareMathOperator{\diff}{d\!}
\DeclareMathOperator{\supp}{supp}

\DeclarePairedDelimiter\abs{\lvert}{\rvert}
\DeclarePairedDelimiter\norm{\lVert}{\rVert}
\def\<{\left\langle}
\def\>{\right\rangle}

\title{Duality for outer $L^p_\mu(\ell^r)$ spaces and relation to tent spaces}

\author{Marco Fraccaroli}
\address{Mathematisches Institut, Universit\"at Bonn, Endenicher Allee 60, 53115 Bonn, Germany}
\email{mfraccar@math.uni-bonn.de}


\begin{document}
	
	\date{\today}
	
	\begin{abstract}		
	We prove that the outer $L^p_\mu(\ell^r)$ spaces, introduced by Do and Thiele, are isomorphic to Banach spaces, and we show the expected duality properties between them for $1 < p \leq \infty, 1 \leq r < \infty$ or $p=r \in \{ 1, \infty \}$ uniformly in the finite setting. In the case $p=1, 1 < r \leq \infty$, we exhibit a counterexample to uniformity. We show that in the upper half space setting these properties hold true in the full range $1 \leq p,r \leq \infty$. These results are obtained via greedy decompositions of functions in $L^p_\mu(\ell^r)$. As a consequence, we establish the equivalence between the classical tent spaces $T^p_r$ and the outer $L^p_\mu(\ell^r)$ spaces in the upper half space. Finally, we give a full classification of weak and strong type estimates for a class of embedding maps to the upper half space with a fractional scale factor for functions on $\R^d$.
	\end{abstract}
	
	\maketitle

\section{Introduction}
The $L^p$ theory for outer measure spaces discussed in \cite{MR3312633} generalizes the classical product, or iteration, of weighted $L^p$ quasi-norms. Since we are mainly interested in positive objects, we assume every function to be nonnegative unless explicitly stated. We first focus on the finite setting.

On the Cartesian product $X$ of two finite sets equipped with strictly positive weights $(Y,\mu),(Z,\nu)$, we can define the classical product, or iterated, $L^\infty L^r, L^pL^r$ spaces for $0 < p , r < \infty$ by the quasi-norms 
\begin{align*}
\norm{f}_{L^\infty((Y,\mu),L^r(Z,\nu))} &= \sup_{y \in Y} (\sum_{z \in Z} \nu(z) f(y,z)^r)^{\frac{1}{r}} \\
&= \sup_{y \in Y} (\mu(y)^{-1} \sum_{z \in Z} \omega(y,z) f(y,z)^r)^{\frac{1}{r}}, \\
\norm{f}_{L^p((Y,\mu),L^r(Z,\nu))} &= (\sum_{y \in Y} \mu(y) (\sum_{z \in Z} \nu(z) f(y,z)^r)^{\frac{p}{r}})^{\frac{1}{p}},
\end{align*}
where we denote by $\omega=\mu \otimes \nu$ the induced weight on $X$. In both cases, the inner $L^r$ quasi-norm may be replaced by an $L^\infty$ norm as well. For $1 \leq p,r \leq \infty$, the objects defined in the display are in fact norms.

We first generalize this setting by making the fibers $Z$ dependent on $y$. Instead of a Cartesian product structure on $X$, we assume $X$ to be partitioned into a collection of subsets $X_y$ parametrized by the elements of $Y$, and let the inner summation run over $x \in X_y$. 

The $L^p$ spaces associated with an outer measure space $(X,\mu)$, or outer $L^p$ spaces, further generalize this construction. The idea is to replace a partition by a covering. An \emph{outer measure} $\mu$ on $X$ is a monotone, subadditive function from $\mathcal{P}(X)$, the power set of $X$, to the extended positive half-line, attaining the value $0$ on the empty set. A standard way to generate an outer measure is via a pre-measure $\sigma$, a function from a collection of subsets $\mathcal{E} \subseteq \mathcal{P}(X)$ to the positive half-line, by means of covering an arbitrary subset of $X$ by elements of $\mathcal{E}$. Namely, for every $A \subseteq X$, we define
\begin{equation} \label{eq:outer_measure_from_pre_measure}
	\mu(A) = \inf\{ \sum_{E \in \mathcal{E}'} \sigma(E) \colon \mathcal{E}' \subseteq \mathcal{E}, A \subseteq \bigcup_{E \in \mathcal{E}'} E \},
\end{equation}
with the understanding that an empty sum is $0$ and that if $A$ is not covered by $\mathcal{E}$, then the infimum is $\infty$. In general, an outer measure need not generate an interesting measure by restriction to the Carath\'{e}odory measurable sets, for example if they are only the empty set and $X$.
Finally, for $\omega$ a strictly positive weight on $X$, $0< r < \infty$, let $\ell^r$ be the function from the Cartesian product of $\mathcal{B}(X)$, the set of Borel measurable functions on $X$, and $\mathcal{P}(X)$, the power set of $X$, defined by
\begin{equation} \label{eq:size}
\ell^r(f)(A) = \ell^r(f,A) = (\mu(A)^{-1} \sum_{x \in A} \omega(x) f(x)^r)^{\frac{1}{r}}.
\end{equation}
The reader familiar with the theory of outer $L^p$ spaces developed in \cite{MR3312633} can recognize that $\ell^r$ is a \emph{size}.

For $0<p,r<\infty$, we can define the outer $L^\infty_\mu(\ell^r),L^p_\mu(\ell^r),L^{p,\infty}_\mu(\ell^r)$ spaces by the quasi-norms
\begin{align}
\label{eq:outer_L_infty_quasi_norm}
\norm{f}_{L^\infty_\mu(\ell^r)} &= \norm{f}_{L^{\infty,\infty}_\mu(\ell^r)} = \sup_{A \subseteq X} \ell^r(f)(A) = \sup_{A \subseteq X} (\mu(A)^{-1} \sum_{x \in A} \omega(x) f(x)^r)^{\frac{1}{r}}, \\
\label{eq:outer_L_p_quasi_norm}
\norm{f}_{L^p_\mu(\ell^r)} &= ( \int_{0}^\infty p \lambda^p \inf \{ \mu(A) \colon A \subseteq X, \norm{f 1_{A^c}}_{L^\infty(\ell^r)} \leq \lambda \} \frac{\diff \lambda}{\lambda})^{\frac{1}{p}}, \\
\label{eq:outer_L_p_infty_quasi_norm}
\norm{f}_{L^{p,\infty}_\mu(\ell^r)} &= ( \sup_{\lambda > 0} \lambda^p \inf \{ \mu(A) \colon A \subseteq X, \norm{f 1_{A^c}}_{L^\infty(\ell^r)} \leq \lambda \} )^{\frac{1}{p}},
\end{align}
where in all the cases the inner $L^r$ quasi-norm may be replaced by an $L^\infty$ norm as well. If the outer measure is generated via a pre-measure, in \eqref{eq:outer_L_infty_quasi_norm} it is enough to take the supremum over the elements of $\mathcal{E}$. The integral in \eqref{eq:outer_L_p_quasi_norm} is reminiscent of the layer-cake representation for the classical $L^p$ norm on a measure space. The subtle point of the theory of outer $L^p$ spaces discussed in \cite{MR3312633} we want to stress is that, in general, the infima in the last display do not stand for outer measures of super level sets for $f$, due to the $L^r$ averaging interplay between $\mu$ and $\omega$. The novelty of the outer $L^p$ spaces consists of allowing for a different way to evaluate the magnitude of a function to define the level sets, in our case the $L^r$ averages, rather than the $L^\infty$ norm. When $r=\infty$, the infimum specializes to the classical concept of the outer measure of the set where $f$ is strictly greater than $\lambda$, namely $\mu(\{ f > \lambda\})$, and the $L^p$ quasi-norm becomes a Choquet integral, but in general there is no relation between the two objects. To shorten the notation, we drop the subscript $\mu$ in $L^p_\mu(\ell^r)$ and we refer to the outer $L^p$ spaces with the symbol $L^p(\ell^r)$. Moreover, we denote the infima in \eqref{eq:outer_L_p_quasi_norm} and \eqref{eq:outer_L_p_infty_quasi_norm} associated with $f,\lambda$ by
\begin{equation} \label{eq:super_level}
\mu(\ell^r (f) > \lambda),
\end{equation}
and we refer to it as the \emph{super level measure}. 

In the first part of this paper, we further develop the theory of outer $L^p$ spaces. In \cite{MR3312633}, the focus was put on the real interpolation features of the outer $L^p$ spaces, such as Marcinkiewicz interpolation and H\"{o}lder's inequality, while other aspects of the theory of these spaces remained untouched. For example, whether the outer $L^p$ quasi-norms are equivalent to norms, and therefore the outer $L^p$ spaces are isomorphic to Banach spaces, or whether they can be recovered as a supremum of a pairing with functions in an appropriate outer $L^{p'}$ space. The first novelty of the paper is to establish the expected properties for the outer $L^p$ spaces where the size is defined by an $L^r$ norm. They follow by the sharpness of the H\"{o}lder's inequality in the sense of the following inequality,
\begin{equation} \label{eq:norm_realization}
	\norm{f}_{L^p(\ell^r)} \leq C \sup_{\norm{g}_{L^{p'}(\ell^{r'})} = 1} \norm{fg}_{L^1(X,\omega)},
\end{equation}
where the constant $C$ is independent of $f \in L^p(\ell^r)$, and $L^1(X,\omega)$ stands for the classical $L^1$ space on $X$ with the measure associated with the weight $\omega$.
\begin{theorem} \label{thm:collapsing_Holder_triangular_finite}
Let $0<p,r \leq \infty$. There exists a constant $C=C(p,r)$ such that, for every finite set $X$, outer measure $\mu$, and strictly positive weight $\omega$, the following properties hold true.
\begin{enumerate} [(i)]
\item For $0 < p=r \leq \infty$, for every $f \in L^p(\ell^p)$,
\begin{equation*}
	\frac{1}{C} \norm{f}_{L^p(X,\omega)} \leq \norm{f}_{L^p(\ell^p)} \leq C \norm{f}_{L^p(X,\omega)}.
\end{equation*}
\item For $1 < p \leq \infty, 1 \leq r < \infty$ or $p=r \in \{1,\infty \}$, for every $f \in L^p(\ell^r)$,
\begin{equation*}
	\frac{1}{C} \sup_{\norm{g}_{L^{p'}(\ell^{r'})} = 1} \norm{fg}_{L^1(X,\omega)} \leq \norm{f}_{L^p(\ell^r)} \leq C \sup_{\norm{g}_{L^{p'}(\ell^{r'})} = 1} \norm{fg}_{L^1(X,\omega)}. 
\end{equation*}
\item For $1 < p \leq \infty, 1 \leq r < \infty$ or $p=r \in \{ 1, \infty \}$, for every $\{ f_n \}_{n \in \N } \subseteq L^p(\ell^r)$,
\begin{equation*}
	\norm{\sum_{n \in \N} f_n}_{L^p(\ell^r)} \leq C \sum_{n \in \N} \norm{f_n}_{L^p(\ell^r)}.
\end{equation*}
\end{enumerate}
Therefore, for $1 < p \leq \infty, 1 \leq r < \infty$ or $p=r \in \{ 1 , \infty \}$, the outer $L^p(\ell^r)$ quasi-norm is equivalent to a norm, the outer $L^p(\ell^r)$ space is isomorphic to a Banach space, and it is the K\"{o}the dual space of the outer $L^{p'}(\ell^{r'})$ space.
\end{theorem}
The main point of the theorem is the uniformity of the constant in $(X,\mu,\omega)$. In fact, for every fixed finite setting, both statements in $(ii),(iii)$ are verified by a certain constant also for $p=1, 1 < r \leq \infty$ or $1<p < \infty, r=\infty$, and hence the final considerations of the theorem hold true as well. However, for $p=1, 1 < r \leq \infty$, the constant is not uniform in $(X,\mu,\omega)$, and we exhibit a counterexample in Lemma \ref{thm:no_uniformity_finite}. For $1<p < \infty, r=\infty$, the question about uniformity remains open. The uniformity of the constant suggests that if an infinite setting is suitably approximated by finite restrictions, the same results could possibly be obtained through a limiting process.

There is a slight abuse in the use of the term K\"{o}the dual space in the statement of Theorem \ref{thm:collapsing_Holder_triangular_finite}, since this object is in general defined for Banach function spaces. A \emph{Banach function space}, or \emph{K\"{o}the function space}, $(\mathcal{L}, \norm{\cdot}_{\mathcal{L}})$ on a $\sigma$-finite measure space $(X,\widetilde{\omega})$ is a Banach space of measurable functions containing all the simple functions and such that if $f$ is a measurable function with absolute value bounded $\widetilde{\omega}$-almost everywhere by $g \in \mathcal{L}$, then $f \in \mathcal{L}$ with norm bounded by that of $g$. The \emph{K\"{o}the dual space}, or \emph{associate space}, of $\mathcal{L}$ is then defined as the space of measurable functions such that the $L^1(X,\widetilde{\omega})$ pairing with every element of $\mathcal{L}$ is finite, endowed with the norm of the dual space, see for example \cite{MR928802,MR540367}. In our setting, we have both a measure associated with the weight $\omega$ and an outer measure $\mu$ on $X$. Although it is not clear whether a priori the simple functions with respect to $\omega$ belong to the outer $L^p(\ell^r)$ space, it is straight-forward to check that the simple functions with respect to $\mu$ belong to $L^p(\ell^r)$. Therefore, with a slight abuse of terminology, we extend the definition of the K\"{o}the duality to the outer $L^p(\ell^r)$ spaces with respect to the $L^1(X,\omega)$ pairing.

The first inequalities of both statements in $(i),(ii)$ were already proved as consequences of more general results obtained in \cite{MR3312633, Uraltsev}, see Proposition \ref{thm:ell^1_domination} and Proposition \ref{thm:Holder_inequality} in the Appendix of the present paper. It would be interesting to investigate whether, for example, the outer $L^p$ spaces are isomorphic to Banach spaces in the level of generality discussed in \cite{MR3312633} and recalled in the Appendix.

The outer $L^p$ spaces were introduced in \cite{MR3312633} with a specific infinite setting in mind. The purpose was to formalize a paradigm in proving the boundedness of modulation invariant operators in time-frequency analysis, when the underlying set is the Cartesian product of the upper half plane with the real line. The bound on the operator is obtained by a two-step program consisting of an outer H\"{o}lder inequality followed by estimates on certain embedding maps from classical to outer $L^p$ spaces. This is for example the case of the bilinear Hilbert transform in \cite{ 2019arXiv190906416A,  MR3829613, MR3312633}, the variational Carleson operator in \cite{MR3829751, 2016arXiv161007657U}, the variational bilinear iterated Fourier inversion operator in \cite{MR3596720}, a family of trilinear multiplier forms with singularity over a one-dimensional subspace in \cite{MR3873113}, and the uniform bilinear Hilbert transform in \cite{Warchalski}. Analogous applications of the outer $L^p$ spaces framework in other settings with different geometries can be found in \cite{2019arXiv190508681A}, \cite{MR3841536}, \cite{MR3875242}, \cite{MR3312633}, \cite{mirek2017local}, \cite{MR3406523}.

In fact, it was pointed out in \cite{MR3312633} that the same two-step program recovers some results of classical non-modulation invariant Calder\'{o}n-Zygmund theory, as detailed for example in \cite{MR0290095, MR1232192}, when the underlying set is the upper half space. In particular, let $X$ be $\R^d \times (0, \infty)$, $\mathcal{D}$ be the collection of open dyadic cubic boxes with sides parallel to the axes and base on $\R^d$, and $\sigma$ the classical volume of the base of the box. Let $\mu$ be the outer measure on $X$ associated with $(\mathcal{D}, \sigma)$ as in \eqref{eq:outer_measure_from_pre_measure}, and $\omega(y,t)$ be the weight $t^{-1}$, where $y \in \R^d, t \in (0,\infty)$. In this infinite setting, we prove the analogous statement of Theorem \ref{thm:collapsing_Holder_triangular_finite}. The properties $(ii),(iii)$ hold even in the endpoint case $p=1$.
\begin{theorem} \label{thm:collapsing_Holder_triangular_upper_half_space}
	Let $(X,\mu,\omega)$ be the upper half space setting just described, $0<p,r \leq \infty$. There exists a constant $C=C(p,r)$ such that the analogous properties stated in Theorem \ref{thm:collapsing_Holder_triangular_finite} hold true in the following range, property $(i)$ for $0 < p=r \leq \infty$, properties $(ii),(iii)$ for $1 \leq p , r \leq \infty$.
	
	Therefore, for $1 \leq p,r \leq \infty$, the outer $L^p(\ell^r)$ quasi-norm is equivalent to a norm, the outer $L^p(\ell^r)$ space is isomorphic to a Banach space, and it is the K\"{o}the dual space of the outer $L^{p'}(\ell^{r'})$ space.
\end{theorem}

In the upper half space setting there are already classical spaces with a different iterated $L^pL^r$ structure, namely the tent spaces introduced in \cite{MR729344, MR791851}, which have been thoroughly studied and used in the literature. Let $\Gamma(x)$ be the cone with vertex in $x \in \R^d$, $T(x,s)$ be the tent over the ball in $\R^d$ centred in $x$ with radius $s$,
\begin{align*}
	\Gamma(x)& = \{ (y,t) \in \R^d \times (0,\infty) \colon \abs{x-y} < t \},\\
	T(x,s)& = \{ (y,t) \in \R^d \times (0,\infty) \colon \abs{x-y} < s-t \}.
\end{align*}
For $0 < p < \infty, 0 < r \leq \infty$, let
\begin{equation} \label{eq:tent_space_1}
	\begin{split}
A_r(f)(x) &= \norm{f}_{L^r(\Gamma(x),\diff y \frac{\diff t}{t^{d+1}})}, \\
\norm{f}_{T^p_r} &= \norm{A_r (f)}_{L^p(\R^d,\diff x)}.
\end{split}
\end{equation} 
For $p=\infty, 0 < r \leq \infty$, let
\begin{equation} \label{eq:tent_space_2}
	\begin{split}
C_r(f)(x) &= \sup_{s \in (0, \infty)} \norm{f}_{L^r(T(x,s), \omega)}, \\
\norm{f}_{T^\infty_r} &= \norm{C_r (f)}_{L^\infty(\R^d,\diff x)}.
\end{split}
\end{equation}
For $0 < p , r \leq \infty$, the tent space $T^p_r$ is defined by the $T^p_r$ quasi-norm. Sometimes in the literature an additional continuity condition is assumed on functions in $T^p_\infty$, see for example \cite{MR791851}, but we do not, in order to preserve a uniformity in the definition of the spaces. For $1 \leq p,r \leq \infty$, the quasi-norms defined in the last two displays are in fact norms.

In \cite{MR3312633}, it was noted that in the upper half space setting, outer $L^p$ spaces can be used in the same spirit of tent spaces in order to prove classical estimates for paraproducts and $T(1)$ theorems. The third result of this paper is to establish the equivalence between the outer $L^p(\ell^r)$ spaces and the tent spaces $T^p_r$ conjectured since the publication of \cite{MR3312633} but never formally established.
\begin{theorem} \label{thm:equivalence_tent_outer_spaces}
	For $0 < p , r \leq \infty$, there exists a constant $C=C(p,r)$ such that, for every $f \in L^p(\ell^r)$,
	\begin{equation*}
		\frac{1}{C} \norm{f}_{T^p_r} \leq \norm{f}_{L^p(\ell^r)} \leq C  \norm{f}_{T^p_r}. 
	\end{equation*}
	Moreover, we have $L^p(\ell^r) = T^p_r$.
\end{theorem}
It is worth noting that while the tent spaces require to pass from cones to tents in order to define $T^\infty_r$, the definition of the outer $L^p(\ell^r)$ spaces always relies on the boxes, or equivalently on the tents.

In the second part of the paper, we turn our focus to embedding maps of functions on $\R^d$ to the upper half space $\R^d\times (0, \infty)$. These embeddings are obtained by pairing a function on $\R^d$ with translated and dilated versions of a given test function. More precisely, given a test function $\phi$ satisfying certain boundedness and decay properties, we define, for every locally integrable function $f$ on $\R^d$, the embedded function $F_\phi(f)$ on $\R^d \times (0 ,\infty)$ by
\begin{equation} \label{eq:embedding_map}
	F_\phi(f)(y,t) = \int_{\R^d} f(x) t^{-d} \phi(t^{-1}(y-x)) \diff x.
\end{equation}
A prominent example of such an embedding is the harmonic extension of a function on $\R^d$ to the upper half space, where $\phi$ is the Poisson kernel. The interest in embedding maps is part of the aforementioned two-step program to prove the boundedness of operators in Calder\'{o}n-Zygmund theory.

We study continuous inclusions between outer $L^p$ spaces in the upper half space and continuous embeddings from classical $L^p$ spaces on $\R^d$ to outer $L^p$ spaces in this setting. We start with an improvement over a previous result on Hardy-Littlewood-Sobolev inclusions between tent spaces in \cite{MR3750310}. We obtain the boundedness of the map
\begin{equation*}
T^p_{r_1} \hookrightarrow T^q_{r_2}, f \mapsto t^{\frac{d}{p}-\frac{d}{q}} f,
\end{equation*} 
for $0 < p < q \leq \infty, 0 < r_2 \leq r_1 \leq \infty$, or equivalently the same statement for outer $L^p(\ell^r)$ spaces. The improvement over the result in \cite{MR3750310} consists of allowing for $r_1$ to be strictly greater than $r_2$. 

These inclusions allow to recover strong type $(p,q)$ estimates for the embedding maps with a fractional scale factor
\begin{equation*}
	L^p(\R^d) \hookrightarrow L^q(\ell^r), f \mapsto t^{\frac{d}{p}-\frac{d}{q}} F_\phi(f),
\end{equation*}
for $0 < p < q \leq \infty, 0 < r \leq \infty$ from the ones for $p=q, r=\infty$. The fourth result of the paper is then the full classification of all positive and negative results regarding strong and weak type estimates for a family of embedding maps with a fractional scale factor in Theorem \ref{thm:embedding}. More precisely, for $\varepsilon > 0,f \in \mathcal{S}(\R^d)$, let the embedded function $F_\varepsilon(f)= F(f)$ be defined by
\begin{equation*} \label{eq:embedding_family_map}
F(f)(y,t) = \sup_{\phi} F_\phi(f)(y,t),
\end{equation*}
where the supremum is taken over the set of functions $\phi$ such that 
\begin{equation} \label{eq:decay_condition}
\abs{\phi(z)} \leq (1+\abs{z})^{-d-\varepsilon}.
\end{equation}

With respect to the strong type estimates, we extract the following statement from Theorem \ref{thm:embedding}.
\begin{theorem}
	Let 
	\begin{equation} \label{eq:zerocase_intro}
	1  \leq p,q \leq \infty,  0< r \leq \infty.
	\end{equation}
	Then, for $(p,q,r)$ satisfying one of the following conditions
	\begin{equation} \label{eq:case_intro} 
	\begin{gathered}
	1 < p < q \leq \infty, 0 < r \leq \infty, \\
	1 < p = q \leq \infty, r = \infty, \\
	p = 1, q = \infty, 0<r\leq \infty,
	\end{gathered}
	\end{equation}
	there exists a constant $C=C(p,q,r,d,\varepsilon)$ such that, for every $f \in L^p(\R^d)$,
	\begin{equation*}
	\norm{t^{\frac{d}{p}-\frac{d}{q}} F (f)}_{L^{q}(\ell^r)} \leq C \norm{f}_{L^p(\R^d)}.
	\end{equation*}
	For all the triples $(p,q,r)$ satisfying \eqref{eq:zerocase_intro} but none of the conditions in \eqref{eq:case_intro}, no strong type $(p,q)$ estimate holds. 
\end{theorem}
It is worth noting that the strong type $(1, \infty)$ estimates hold for $0< r \leq \infty$, even if  for $r=\infty$ only the weak type $(1,1)$ estimate holds. Moreover, in the endpoint $p=q=1, r=\infty$, we prove in Proposition \ref{thm:embedding_Hardy_space} a substitute of the strong type $(1,1)$ estimate, namely the boundedness of the embedding map
\begin{equation*}
	H^1(\R^d) \hookrightarrow L^1(\ell^\infty), f \mapsto F_\varphi(f),
\end{equation*}
for $\varphi \in \mathcal{S}(\R^d)$. 

We conclude the paper with some applications of these embedding theorems yielding alternative proofs of classical results such as the Hardy-Littlewood-Sobolev inequality, and the Gagliardo-Nirenberg-Sobolev inequality up to the endpoint in the spirit of the aforementioned two-step program.  

\subsection*{Guide to the paper}
In Section 2 we state and prove two crucial preparatory decomposition results for functions in the outer $L^p(\ell^r)$ spaces in both finite and upper half space settings. 
We use them to prove Theorem \ref{thm:collapsing_Holder_triangular_upper_half_space} and Theorem \ref{thm:collapsing_Holder_triangular_finite} in Section 3. Moreover, in Lemma \ref{thm:no_uniformity_finite}, we provide a counterexample to the uniformity of the statements in $(ii),(iii)$ in Theorem \ref{thm:collapsing_Holder_triangular_finite} for $p=1, 1 < r \leq \infty$. 
In Section 4 we prove Theorem \ref{thm:equivalence_tent_outer_spaces}.
In Section 5, Theorem \ref{thm:inclusion_tent_spaces}, we improve over the result of Amenta on Hardy-Littlewood-Sobolev inclusions between tent spaces.
In Section 6, Theorem \ref{thm:embedding}, we prove a full classification of all positive and negative results regarding strong and weak type estimates for a family of embedding maps with a fractional scale factor from classical $L^p$ spaces on $\R^d$ to outer $L^p(\ell^r)$ spaces. Moreover, in Proposition \ref{thm:embedding_Hardy_space} we prove the boundedness of the embedding map defined by a test function $\varphi \in \mathcal{S}(\R^d)$ from $H^1(\R^d)$ to the outer $L^1(\ell^\infty)$ space.
We use the strong type estimates from both results to prove the Hardy-Littlewood-Sobolev inequality, and the Gagliardo-Nirenberg-Sobolev inequality up to the endpoint in the spirit of the aforementioned two-step program in Section 7.
Finally, in Section 8, the Appendix, we review the definitions and recall some results of the theory of outer $L^p$ spaces in the level of generality discussed in \cite{MR3312633}.

\section*{Acknowledgements}
The author gratefully acknowledges financial support by the CRC 1060 \emph{The Mathematics of Emergent Effects} at the University of Bonn, funded through the Deutsche Forschungsgemeinschaft. The author is thankful to Alex Amenta, Christoph Thiele and Gennady Uraltsev for helpful comments, suggestions and corrections that improved the exposition of the material, and for their support.

\section{Decompositions for outer $L^p(\ell^r)$ spaces}
In this section we state and prove two crucial preparatory decomposition results for functions in the outer $L^p(\ell^r)$ spaces in both finite and upper half space settings, used in proving Theorem \ref{thm:collapsing_Holder_triangular_finite} and Theorem \ref{thm:collapsing_Holder_triangular_upper_half_space}, respectively. Both consist of a recursive greedy selection algorithm that provides a sequence of maximal disjoint subsets of $X$ exhausting the elements of $\mathcal{P}(X)$ where the quantity defined in \eqref{eq:size} is in the interval $[2^k,2^{k+1}), k \in \Z$. This property guarantees not only an upper bound but also a lower bound on the super level measure in \eqref{eq:super_level} at level $\lambda=2^k,k \in \Z$, in terms of the outer measures of the selected subsets, thus providing a concrete substitute for it. Without loss of generality, we can restrict our attention only to these levels. In fact, we can replace the integral in \eqref{eq:outer_L_p_quasi_norm} with an equivalent discrete version, namely
\begin{equation*}
	(\sum_{k \in \Z} 2^{kp} \mu(\ell^r(f) > 2^k) )^{\frac{1}{p}},
\end{equation*}   
due to the monotonicity in $\lambda$ of the super level measure of a fixed function. This quantity is no longer homogeneous in $f$, hence it is not a quasi-norm, but the discrete levels fit better the recursive process we want to describe.

In the finite setting we do not have to worry about the well-definedness of the selection process, since at each step only finitely many choices are available. Again, we stress that the main point in this case is the uniformity of constants in $(X,\mu,\omega)$. On the contrary, in the upper half space setting we need to be slightly more careful in the definition of the recursive algorithm, accounting for the non-finiteness of the selection process. On the other hand, due to the geometry of the outer measure space, we can get an improved version of the decomposition result. First, we can extend it to the case $r=\infty$, which is not included in the finite setting. Second, the decomposition of a function in the outer $L^1(\ell^r)$ space, for $1< r \leq \infty$, is subtly more efficient for our purpose, as will be clarified in Remark \ref{rmk:gain_of_the_endpoint}.

We start with the finite setting. Let $X$ be a finite set, $\mu$ an outer measure, $\omega$ a strictly positive weight. Without loss of generality, we can assume that $\mu$ attains a strictly positive finite value on every singleton in $\mathcal{P}(X)$. If there exists $x \in X$ such that $\mu(\{x\}) \in \{0,\infty\}$, then we can subtract $x$ from the set, as no nontrivial contribution to the outer $L^p(\ell^r)$ spaces comes from these singleton. In fact, in the first case the outer $L^p(\ell^r)$ quasi-norm of a function $f$ is the same as that of $f 1_{X \setminus \{x\}}$, and in the second the quasi-norm is infinite as soon as $f(x)$ is different from $0$. As a consequence, $\mu(X)$ is finite.

We have the following uniform decomposition result for functions in the outer $L^p(\ell^r)$ spaces defined by \eqref{eq:outer_L_p_quasi_norm}.
\begin{proposition} \label{thm:atomic_decomposition_finite}
	Let $0 < p ,r < \infty$. There exists a constant $C=C(p,r)$ such that, for every finite set $X$, outer measure $\mu$, and strictly positive weight $\omega$, the following property holds. For $f \in L^p(\ell^r)$, there exists a sequence of sets $\{ E_k \colon k \in \Z \} \subseteq \mathcal{P}(X)$ such that if
	\begin{equation*}
		F_{k} = \bigcup_{ l \geq k } E_{l},
	\end{equation*}
	then, for every $k \in \Z$,
	\begin{gather} 
		\label{eq:superlevel_finite}
		\ell^r (f 1_{F_{k+1}^c})(E_k) > 2^k, \ \ \ \ \text{when $E_k \neq \emptyset$,} \\
		\label{eq:maximal_choice_finite}
		\norm{f 1_{F_k^c}}_{L^\infty(\ell^r)} \leq 2^k, \\
		\label{eq:covering_finite}
		\mu(\ell^r (f) > 2^k) \leq \sum_{ l \geq k } \mu(E_l), \\
		\label{eq:optimal_covering_finite}
		\mu (E_k) \leq C \mu(\ell^r (f) > 2^{k-1}).
	\end{gather}
\end{proposition}
\begin{proof}
	First, we observe qualitatively that by outer H\"{o}lder's inequality, Proposition \ref{thm:Holder_inequality}, we have $L^p(\ell^r) \subseteq L^\infty(\ell^r)$, because $\mu(X)$ is finite.
	
	We define $E_k$ by backward recursion on $k \in \Z$. For $k$ large enough such that
	\begin{equation*}
		\norm{f}_{L^\infty(\ell^r)} \leq 2^k,
	\end{equation*}
	we set $E_k$ to be empty. Now fix $k$ and assume we have selected $E_l$ for $l > k$. In particular, $F_{k+1}$ is already well-defined. If there exists a set $A \subseteq X$ such that
	\begin{equation} \label{eq:selection_condition}
		\ell^r(f 1_{F_{k+1}^c})(A) > 2^k,
	\end{equation}
	then we choose such a set $A$ to be $E_k$, making sure that
	\begin{equation} \label{eq:stopping_condition}
		\norm{ f 1_{(A \cup F_{k+1})^c}}_{L^\infty(\ell^r)} \leq 2^k.
	\end{equation}
	In fact, if there exists a set $B \subseteq X$ such that
	\begin{equation*}
		\ell^r(f 1_{(A \cup F_{k+1})^c})(B) > 2^k,
	\end{equation*}
	then by the subadditivity of the outer measure, we have
	\begin{equation*}
		\ell^r(f 1_{ F_{k+1}^c})(A \cup B) > 2^k.
	\end{equation*}
	Due to the finiteness of $X$, the condition \eqref{eq:stopping_condition} can be achieved in finitely many steps. If no $A$ satisfying \eqref{eq:selection_condition} exists, we set $E_k$ to be empty, and proceed the recursion with $k-1$.
	
	By construction, we have \eqref{eq:superlevel_finite} for every nonempty selected set $E_k$, \eqref{eq:maximal_choice_finite} and \eqref{eq:covering_finite} for every $k \in \Z$. 
	
	We observe that for every $k$ such that $2^k$ is greater than the $L^\infty(\ell^r)$ quasi-norm of $f$, the statement \eqref{eq:optimal_covering_finite} is true. To prove \eqref{eq:optimal_covering_finite} for any other $k$, let $A_{k-1}$ be a set witnessing the super level measure at level $2^{k-1}$. In particular,
	\begin{equation*}
		\mu(\ell^r(f)>2^{k-1}) = \mu(A_{k-1}).
	\end{equation*}
	By \eqref{eq:maximal_choice_finite} for $k+1$, we have
	\begin{equation} \label{eq:equation_finite}
		\mu(A_{k-1}) \geq 2^{-r(k+1)} \sum_{A_{k-1} \setminus F_{k+1}} \omega(x)  f(x)^r .
	\end{equation}
	By the definition of $A_{k-1}$ and $E_{k}$, we have
	\begin{align*}
		\sum_{E_{k} \setminus A_{k-1}} \omega(x)  f(x)^r &\leq 2^{r(k-1)} \mu(E_{k}), \\
		\sum_{E_{k} \setminus F_{k+1}}  \omega(x) f(x)^r  &> 2^{rk} \mu(E_{k}),
	\end{align*}
	hence
	\begin{equation*}
		\sum_{(A_{k-1} \cap E_{k}) \setminus F_{k+1}}  \omega(x) f(x)^r > C 2^{r(k-1)} \mu(E_{k}).
	\end{equation*}
	Combining this with \eqref{eq:equation_finite} gives
	\begin{equation*}
		\mu(\ell^r(f)>2^{k-1}) \geq C \mu(E_{k}),
	\end{equation*}
	concluding the proof of \eqref{eq:optimal_covering_finite} for the given $k$.
\end{proof}

Now we move to the upper half space setting. Let $X$ be the upper half space and $\mu$ the outer measure generated by the pre-measure $\sigma$ on $\mathcal{D}$, the collection of open dyadic boxes in the upper half space, as in \eqref{eq:outer_measure_from_pre_measure}. In particular, 
\begin{equation} \label{eq:upperspace_dyadic_setting}
	\begin{split}
		X &= \R^d \times (0,\infty), \\
		\mathcal{D} &= \{ (x,0) + (0,2^j)^{d+1} \colon x \in 2^{j}\Z^d, j \in \Z \}, \\
		\sigma(E) &= \abs{B(E)}, \ \ \ \ \text{for every $E \in \mathcal{D}$,}\\
		\omega(y,t) &= t^{-1},
	\end{split}
\end{equation}
where $B(E)$ is the base in $\R^d$ of the dyadic box $E \in \mathcal{D}$.

In the following statement, the elements of a double sequence are parametrized by a pair $(k,n)$, for $ k \in \Z, n \in \N_k$, where $\N_k$ is either the set of positive natural numbers or a possibly empty finite initial string of positive natural numbers. We consider the lexicographic order of such pairs as follows: $(l,m) < (k,n)$ if either $l>k$, or $l = k$ and $m<n$.

We have the following decomposition result for functions in the intersection between the outer $L^p(\ell^r)$ and $L^\infty(\ell^r)$ spaces defined by \eqref{eq:outer_L_p_quasi_norm} and \eqref{eq:outer_L_infty_quasi_norm}, respectively.
\begin{proposition} \label{thm:atomic_decomposition_upperspace_dyadic}
	Let $0 < p < \infty, 0 < r \leq \infty$. There exists a constant $C = C(p,r)$ such that the following property holds. For $f \in L^p(\ell^r) \cap L^\infty(\ell^r)$, there exists a double sequence of dyadic boxes $\{ E_{k,n} \colon k \in \Z, n \in \N_k \} \subseteq \mathcal{D}$ such that if
	\begin{equation*}
		\begin{gathered}
		F_k = \bigcup_{n \in \N_k} F_{k,n}, \\
		F_{k,n} = F_{k,n-1} \cup E_{k,n}, \\
		F_{k,0}  = \bigcup_{i \in I_k} Q_i,
		\end{gathered}
	\end{equation*}
	where $\{Q_i \colon i \in I_k \} \subseteq \mathcal{D}$ is the collection of maximal dyadic boxes such that
	\begin{equation} \label{eq:doubling_condition_upperspace_dyadic}
		\abs{B(Q_i)} \leq 2 \abs{B(Q_i) \cap \bigcup_{(l,m) \colon l > k} B(E_{l,m})} ,
	\end{equation}
	then, for every $k \in \Z, n \in \N_k$,
	\begin{gather} 
		\label{eq:superlevel_upperspace_dyadic}
		\ell^r (f 1_{F_{k,n-1}^c})(E_{k,n}) > 2^k,  \ \ \ \ \text{when $E_{k,n} \neq \emptyset$,} \\
		\label{eq:maximal_choice_upperspace_dyadic}
		\norm{f 1_{F_k^c}}_{L^\infty(\ell^r)} \leq 2^{k}, \\
		\label{eq:covering_upperspace_dyadic}
		\mu(\ell^r (f) > 2^{k}) \leq C \sum_{ (l,m) \colon l \geq k } \sigma(E_{l,m}), \\
		\label{eq:optimal_covering_upperspace_dyadic}
		\sum_{n \in \N_k} \sigma (E_{k,n}) \leq C \mu(\ell^r (f) > 2^{k-1}).
	\end{gather}
	Moreover, the collection $\{ B(E_{k,n}) \colon k \in \Z, n \in \N_k \}$ of the bases of the chosen boxes is $2$-Carleson, i.e. for every dyadic box $E \in \mathcal{D}$
	\begin{equation} \label{eq:sparse_collection_upperspace_dyadic}
		\sum_{(k,n) \colon E_{k,n} \subseteq E} \sigma(E_{k,n}) \leq 2 \sigma(E).
	\end{equation}
\end{proposition}
A dyadic box satisfies the condition in \eqref{eq:doubling_condition_upperspace_dyadic} for a certain $k \in \Z$ when at least half of its base is covered by the bases of the elements of the double sequence selected up to the level $k+1$. 
\begin{proof}
	For $0< r < \infty$, the selection algorithm is analogous to that described in the previous proof. We define $E_{k,n}$ by a double recursion, backward on $k \in \Z$, and, for every fixed $k$, forward on $n \in \N_k$. For $k$ large enough such that
	\begin{equation*}
		\norm{f}_{L^\infty(\ell^r)} \leq 2^k,
	\end{equation*}
	we set $\N_k$ empty. Now fix $(k,n)$ and assume we have selected $E_{l,m}$ for $(l,m) < (k,n)$. In particular, $F_{k,n-1}$ is already well-defined. If there exists a dyadic box $A \in \mathcal{D}$ such that
	\begin{equation} \label{eq:selection_condition_upper_half_space}
		\ell^r(f 1_{F_{k,n-1}^c})(A) > 2^k,
	\end{equation}
	then we choose such a dyadic box $A$ to be $E_{k,n}$, making sure that $\sigma(A)$ is maximal. The maximality of $\sigma(A)$ is achieved in finitely many steps because the set of values of $\sigma$ is discrete and doubling and we have an upper bound on $\sigma(A)$ when $A$ satisfies the condition \eqref{eq:selection_condition_upper_half_space}. In fact, we have
	\begin{equation*}
		\sigma(A) \leq C \mu(\ell^r(f)>2^{k-1}) \leq C 2^{-kp} \norm{f}_{L^p(\ell^r)}^p,
	\end{equation*}
	where in the first inequality we used an argument analogous to that used to prove \eqref{eq:optimal_covering_finite} above. 
	
	If no $A$ satisfying \eqref{eq:selection_condition_upper_half_space} exists, we set $\N_k = \{ 1, \dots, n \}$, $\N_k$ empty if $n=1$, and proceed the recursion with $(k-1,1)$. If for some $k$ we are able to choose $E_{k,n}$ for all $n \in \N$, we fix such $E_{k,n}$ and proceed the recursion with $(k-1,1)$.
	
	By construction, we have \eqref{eq:superlevel_upperspace_dyadic} for every nonempty selected dyadic box $E_{k,n}$.
	
	We prove \eqref{eq:maximal_choice_upperspace_dyadic} and \eqref{eq:optimal_covering_upperspace_dyadic} by backward induction on $k \in \Z$. We observe that for every $k$ such that $2^k$ is greater than the $L^\infty(\ell^r)$ quasi-norm of $f$, both statements are true.
	
	The proof of \eqref{eq:optimal_covering_upperspace_dyadic} for any other $k$ assuming \eqref{eq:maximal_choice_upperspace_dyadic} for $k+1$, which we have by the induction hypothesis, is analogous to that of the corresponding property \eqref{eq:optimal_covering_finite} in the previous result. The minor adjustments concern the fact that we have to deal with a near optimal cover of the super level measure set at level $2^{k-1}$ and a collection $\{E_{k,n} \colon n \in \N_k \}$ instead of the sets $A_{k-1}$ and $E_{k}$.
	
	Now we prove \eqref{eq:maximal_choice_upperspace_dyadic} for every $k$ such that $2^k$ is strictly smaller than the $L^\infty(\ell^r)$ quasi-norm of $f$. If $\N_k$ is finite, then by construction there are no dyadic boxes $A \in \mathcal{D}$ such that
	\begin{equation*}
		\ell^r(f 1_{F_k^c})(A) > 2^k.
	\end{equation*}
	If $\N_k$ is infinite, we observe by \eqref{eq:optimal_covering_upperspace_dyadic}, which we already proved for this $k$, that
	\begin{equation*}
		\sum_{n \in \N_k} \sigma(E_{k,n}) < \infty,
	\end{equation*}
	since $f \in L^p(\ell^r)$. Therefore, $\sigma(E_{k,n})$ tends to zero as $n$ tends to $\infty$. Since each $E_{k,n}$ is chosen to maximize $\sigma(E_{k,n})$, there exists no dyadic box $A \in \mathcal{D}$ which can violate \eqref{eq:maximal_choice_upperspace_dyadic} as such $A$ would contradict the choice of $E_{k,n}$ for sufficiently large $n$. This concludes the proof of \eqref{eq:maximal_choice_upperspace_dyadic} for the given $k$.
	
	With \eqref{eq:maximal_choice_upperspace_dyadic}, we also have \eqref{eq:covering_upperspace_dyadic}. In fact, we have
	\begin{align*}
		\mu(F_k) &\leq \mu(F_{k-1,0}) \\
		&\leq \sum_{i \in I_{k-1}} \abs{B(Q_i)} \\
		&\leq 2 \abs{\bigcup_{i \in I_{k-1}} B(Q_i) \cap \bigcup_{(l,m) \colon l\geq k} B(E_{l,m})} \\
		&\leq C \sum_{(l,m) \colon l \geq k} \sigma(E_{l,m}),
	\end{align*}
	where we used \eqref{eq:doubling_condition_upperspace_dyadic} and the disjointness of the elements of $\{Q_i \colon i \in I_{k-1} \}$ in the third inequality.
	
	For $r= \infty$, the only difference is in the selection of $E_{k,n}$. Fix $(k,n)$ and assume we have selected $E_{l,m}$ for $(l,m) < (k,n)$. If there exists a dyadic box $A = (x,0) + (0,s)^{d+1} \in \mathcal{D}$ such that
	\begin{equation*}
		\ell^\infty(f 1_{F_{k,n-1}^c} 1_{A^+})(A) > 2^k,
	\end{equation*}
	where $A^+ = (x,0)+((0,s)^d \times (s/2,s))$, then we choose such a dyadic box $A$ to be $E_{k,n}$, making sure $\sigma(E_{k,n})$ is maximal. The proof of \eqref{eq:superlevel_upperspace_dyadic}, \eqref{eq:maximal_choice_upperspace_dyadic}, \eqref{eq:covering_upperspace_dyadic}, and \eqref{eq:optimal_covering_upperspace_dyadic} then follows in a straight-forward way.
	
	To conclude, we observe that the collection $\{ B(E_{k,n}) \colon k \in \Z, n \in \N_k \}$ is $1/2$-sparse, i.e. one can choose pairwise disjoint measurable sets $\widetilde{B}_{k,n} \subseteq B(E_{k,n})$ with $\abs{\widetilde{B}_{k,n}} \geq \abs{ B(E_{k,n})}/2$. This follows by \eqref{eq:doubling_condition_upperspace_dyadic} and the maximality in the choice of $E_{k,n}$. Therefore, the collection is $2$-Carleson.
\end{proof}

\section{Equivalence with norms}
In this section we prove Theorem \ref{thm:collapsing_Holder_triangular_upper_half_space} and Theorem \ref{thm:collapsing_Holder_triangular_finite}. We start with the upper half space setting. First, we prove property $(i)$. After that, for every $f \in L^p(\ell^r) \cap L^\infty(\ell^r)$, for $1 \leq p,r \leq \infty$, we provide a candidate function $g$ to realize \eqref{eq:norm_realization}, up to normalization of its outer $L^{p'}(\ell^{r'})$ quasi-norm. Upon showing an upper bound on the outer $L^{p'}(\ell^{r'})$ quasi-norm of $g$ and a lower bound on the $L^1(X,\omega)$ norm of $fg$, properties $(ii),(iii)$ follow. 
Then we turn to the finite setting and when possible we follow analogous arguments to prove properties $(i),(ii)$, and $(iii)$. In almost all the definitions and proofs we make use of the decompositions provided by Proposition \ref{thm:atomic_decomposition_upperspace_dyadic} and Proposition \ref{thm:atomic_decomposition_finite}.  
Finally, in Lemma \ref{thm:no_uniformity_finite} we exhibit a counterexample to the uniformity in every finite setting $(X,\mu,\omega)$ of both statements in $(ii),(iii)$ for $p=1, 1 < r \leq \infty$.

We start with the upper half space setting, where $(X,\mu,\omega)$ is the setting described in \eqref{eq:upperspace_dyadic_setting},
\begin{proof}[Proof of Theorem \ref{thm:collapsing_Holder_triangular_upper_half_space}, property $(i)$]
	The case $p= \infty$ follows by definition.
	
	Therefore, we can assume without loss of generality $p=1$, since
	\begin{equation*}
		\norm{f}_{L^p(\ell^p)}^p = \norm{f^p}_{L^1(\ell^1)}.
	\end{equation*}
	For $f \in L^1(\ell^1) \cap L^\infty(\ell^1)$, let $\{E_{k,n}\}$ be the collection of dyadic boxes from Proposition \ref{thm:atomic_decomposition_upperspace_dyadic}. We have
	\begin{align*}
		\norm{f}_{L^1(\ell^1)} &\leq C \sum_{k \in \Z} 2^{k} \mu(\ell^1 (f) >2^k) \\
		&\leq C \sum_{k \in \Z} 2^{k} \sum_{(l,m) \colon l \geq k} \sigma(E_{l,m}) \\
		&\leq C \sum_{l \in \Z} \sum_{m \in \N_l} 2^l \sigma(E_{l,m}) \\
		&\leq C \sum_{l \in \Z} \sum_{m \in \N_l} \norm{f}_{L^1(E_{l,m} \setminus F_{l,m-1}, \omega)} \\
		&\leq C \norm{f}_{L^1(X, \omega)},
	\end{align*}
	where we used \eqref{eq:covering_upperspace_dyadic} in the second inequality, Fubini and the bounds on the geometric series in the third, \eqref{eq:superlevel_upperspace_dyadic} in the fourth, and disjointness of the sets in the fifth.
	
	We note that $f$ vanishes $\omega$-almost everywhere outside the union of all the selected dyadic boxes $\{ E_{k,n} \}$, since $\mathcal{D}$ covers all of $X$. We have
	\begin{align*}
		\norm{f}_{L^1(X, \omega)} &= \sum_{k \in \Z} \sum_{n \in \N_k}  \norm{f 1_{E_{k,n} \setminus F_{k,n-1}}}_{L^1(X, \omega)} + \sum_{k \in \Z} \norm{f 1_{F_{k,0} \setminus F_{k+1}}}_{L^1(X, \omega)} \\
		&\leq \sum_{k \in \Z} \sum_{n \in \N_k}  \norm{f 1_{E_{k,n} \setminus F_{k+1}}}_{L^1(X, \omega)} + \sum_{k \in \Z} \sum_{i \in I_k} \norm{f 1_{Q_{i} \setminus F_{k+1}}}_{L^1(X, \omega)}  \\
		&\leq \sum_{k \in \Z} 2^{k+1} \sum_{n \in \N_k} \sigma(E_{k,n}) + \sum_{k \in \Z} 2^{k+1} \sum_{i \in I_k} \sigma(Q_i) \\
		&\leq \sum_{k \in \Z} 2^{k+1} \sum_{n \in \N_k} \sigma(E_{k,n}) + \sum_{k \in \Z} 2^{k+1} \sum_{(l,m) \colon l > k} \sigma(E_{l,m}) \\
		&\leq \sum_{k \in \Z} 2^{k+1} \sum_{n \in \N_k} \sigma(E_{k,n}) + C \sum_{l \in \Z} 2^{l+1} \sum_{m \in \N_l} \sigma(E_{l,m}) \\
		&\leq C \sum_{k \in \Z} 2^{k} \mu( \ell^1(f) > 2^{k-1}) \\
		&\leq C \norm{f}_{L^1(\ell^1)},
	\end{align*}
	where we used \eqref{eq:maximal_choice_upperspace_dyadic} in the second inequality, \eqref{eq:doubling_condition_upperspace_dyadic} and the disjointness of the dyadic boxes $\{Q_i\}$ in the third, Fubini and the bounds on the geometric series in the fourth, and \eqref{eq:optimal_covering_upperspace_dyadic} in the fifth. 
	
	A standard approximation argument yields the result for arbitrary $f \in L^1(\ell^1)$.
\end{proof}

Now we provide the candidate function $g$ for $f \in L^p(\ell^r) \cap L^\infty(\ell^r)$, for $1 \leq p, r \leq \infty$. We separate into four cases depending on $p$ and $r$.

{\textbf{Case 1: $1 \leq p, r < \infty$.}} For $f \in L^p(\ell^r) \cap L^\infty(\ell^r)$, let $\{E_{k,n}\}$ be the collection from Proposition \ref{thm:atomic_decomposition_upperspace_dyadic}, and define
\begin{equation*} \label{eq:dual_function_upperspace_dyadic}
	g(x,s) = \sum_{k \in \Z } \sum_{n \in \N_k} 2^{k(p-r)} 1_{E_{k,n} \setminus F_{k,n-1}}(x,s) f(x,s)^{r-1}.
\end{equation*}

{\textbf{Case 2: $1 \leq p < \infty$ and $r= \infty$.}} For $f \in L^p(\ell^\infty) \cap L^\infty(\ell^\infty)$, let $\{E_{k,n}\}$ be the collection from Proposition \ref{thm:atomic_decomposition_upperspace_dyadic}, and define
\begin{equation*} \label{eq:dual_function_upperspace_dyadic_r_infty}
	g(x,s) = \sum_{k \in \Z } \sum_{n \in \N_k} 2^{k(p-1)} 1_{\widetilde{E}_{k,n}}(x,s) (\ell^1(1_{\widetilde{E}_{k,n}})(E_{k,n}))^{-1} ,
\end{equation*}
where 
\begin{equation*}
	\widetilde{E}_{k,n} = E_{k,n}^+ \cap \{ f > 2^{k} \},
\end{equation*}
and $E_{k,n}^+$ is the upper half of $E_{k,n}$.

{\textbf{Case 3: $p= \infty$ and $1 \leq r < \infty$.}} For $f \in L^\infty(\ell^r)$, let the dyadic box $E \in \mathcal{D}$ witness the outer $L^\infty(\ell^r)$ quasi-norm of $f$ up to a factor $2$, and define
\begin{equation*} \label{eq:dual_function_upperspace_dyadic_p_infty}
	g(x,s) = 1_E(x,s) f(x,s)^{r-1} .
\end{equation*}

{\textbf{Case 4: $p= r = \infty$.}} For $f \in L^\infty(\ell^\infty)$, let the dyadic box $E \in \mathcal{D}$ witness the outer $L^\infty(\ell^\infty)$ quasi-norm of $f$ up to a factor $2$ in a subset of strictly positive measure in $E^+$, and define
\begin{equation*} \label{eq:dual_function_upperspace_dyadic_p_infty_r_infty}
	g(x,s) = 1_{\widetilde{E}}(x,s) (\ell^1(1_{\widetilde{E}})(E))^{-1},
\end{equation*}
where 
\begin{equation*}
	\widetilde{E} = E^+ \cap \{ f > \norm{f}_{L^\infty(\ell^\infty)}/2 \}.
\end{equation*}

We have the following upper bounds on the outer $L^{p'}(\ell^{r'})$ quasi-norm of $g$, where $g$ is defined according to the four $(p,r)$-dependent cases.
\begin{lemma} \label{thm:upper_bound_norm_realizing_function_upperspace_dyadic}
	{\textbf{Case I: $p=1$ and $ 1 \leq r \leq \infty$.}} We have
	\begin{equation*}
		\norm{g}_{L^{\infty}(\ell^{r'})} \leq C.
	\end{equation*}
	{\textbf{Case II: $1 < p < \infty$ and $1 \leq r \leq \infty$.}} We have
	\begin{equation*}
		\norm{g}_{L^{p'}(\ell^{r'})}^{p'} \leq C \norm{f}_{L^p(\ell^r)}^{p}.
	\end{equation*}
	{\textbf{Case III: $p= \infty$ and $1 \leq r < \infty$.}} We have
	\begin{equation*}
		\norm{g}_{L^{1}(\ell^{r'})} \leq \norm{f}_{L^\infty(\ell^r)}^{r-1} \sigma(E).
	\end{equation*}
	{\textbf{Case IV: $p=r = \infty$.}} We have
	\begin{equation*}
		\norm{g}_{L^{1}(\ell^1)} \leq \sigma(E).
	\end{equation*}
\end{lemma}
\begin{proof}
	{\textbf{Case I: $p=1$ and $ 1 \leq r \leq \infty$.}} Let $1<r < \infty$. For every dyadic box $A \in \mathcal{D}$, we have
	\begin{equation}  \label{eq:size_norm_realizing_function_upperspace_dyadic}
		\begin{split}
			(\ell^{r'}(g)(A))^{r'} &= \frac{1}{\sigma(A)} \sum_{k \in \Z} \sum_{n \in \N_k} 2^{-kr} \int_{A \cap (E_{k,n} \setminus F_{k,n-1})} f(y,t)^r \omega(y,t) \diff y \diff t \\
			&\leq \frac{1}{\sigma(A)} \sum_{k \in \Z} \sum_{n \in \N_k} 2^{-kr} \int_{A \cap (E_{k,n} \setminus F_{k+1})} f(y,t)^r \omega(y,t) \diff y \diff t \\
			&\leq C \frac{1}{\sigma(A)} ( \sigma(A) + \sum_{(k,n) \colon E_{k,n} \subseteq A} \sigma(E_{k,n})) \\
			&\leq C,
		\end{split}
	\end{equation}
	where we used \eqref{eq:maximal_choice_upperspace_dyadic} and the nested structure of $\mathcal{D}$, namely the fact that for $A,B \in \mathcal{D}, A \cap B \neq \emptyset$, then either $A \subseteq B$ or $B \subseteq A$, in the second inequality, and \eqref{eq:sparse_collection_upperspace_dyadic} in the third.
	
	In an analogous way, for every dyadic box $A \in \mathcal{D}$, for $r=\infty$, we have
	\begin{equation*}
		\ell^{1}(g)(A) \leq C, 
	\end{equation*}
	and it is easy to see that, for $r=1$, we have
	\begin{equation*}
		\ell^{\infty}(g)(A) \leq 1. 
	\end{equation*} 
	
	Therefore, for $1 \leq r \leq \infty$, we have
	\begin{equation*}
		\norm{g}_{L^{\infty}(\ell^{r'})} \leq C.
	\end{equation*}
	
	{\textbf{Case II: $1< p <\infty$ and $1 \leq r \leq \infty$.}} Let $1<r<\infty$. For a fixed $k$ and every dyadic box $A \in \mathcal{D}$, we have
	\begin{equation} \label{eq:level_set_norm_realizing_function_upperspace_dyadic_r}
		\begin{split}
			(\ell^{r'}(g 1_{F_k^c})(A))^{r'} &= \frac{1}{\sigma(A)} \sum_{(l,m) \colon l<k} 2^{l(p-r)r'} \int_{A \cap (E_{l,m} \setminus F_{l,m-1})} f(y,t)^r \omega(y,t) \diff y \diff t \\
			&\leq \sum_{l<k} 2^{l(p-r)r'} \frac{1}{\sigma(A)} \int_{A \setminus F_{l+1}} f(y,t)^r \omega(y,t) \diff y \diff t \\
			&\leq c \sum_{l<k} 2^{l(p-r+r-1)r'} \\
			&\leq c 2^{k(p-1)r'},
		\end{split}
	\end{equation}
	where we used \eqref{eq:maximal_choice_upperspace_dyadic} in the second inequality, and the bounds on the geometric series in the third.
	
	In an analogous way, for every dyadic box $A \in \mathcal{D}$, for $r=\infty$, we have
	\begin{equation} \label{eq:level_set_norm_realizing_function_upperspace_dyadic_infty}
		\begin{split}
		\ell^{1}(g 1_{F_k^c})(A) & = \frac{1}{\sigma(A)} \sum_{l<k} \sum_{m \in \N_l}  2^{l(p-1)} \int_{A \cap \widetilde{E}_{l,m}} (\ell^1(1_{\widetilde{E}_{l.m}})(E_{l.m}))^{-1}  \omega(y,t) \diff y \diff t \\
		&\leq \sum_{l<k} 2^{l(p-1)} \frac{1}{\sigma(A)} \sum_{m \colon E_{l,m} \subseteq A} \sigma(E_{l,m}) \\
		&\leq c 2^{k(p-1)}, 
		\end{split}
	\end{equation}
	where we used the disjointness of the elements of $\{E_{l,m} \colon m \in \N_l\}$ due to the maximality in their choice, and the bounds on the geometric series in the second inequality. 
	
	It is easy to see that, for every dyadic box $A \in \mathcal{D}$, for $r=1$, we have
	\begin{equation*}
		\ell^{\infty}(g 1_{F_k^c})(A) \leq 2^{k(p-1)}. 
	\end{equation*} 
	
	As a consequence, for $1 \leq r \leq \infty$, for every dyadic box $A \in \mathcal{D}$, we have
	\begin{equation*}
		\ell^{r'} (g 1_{F_k^c})(A) \leq c 2^{k(p-1)},
	\end{equation*}
	hence
	\begin{equation} \label{eq:upper_bound_level_set}
		\mu( \ell^{r'}(g) > c 2^{k(p-1)}) \leq \mu(F_k) \leq C \sum_{(l,m) \colon l \geq k} \sigma(E_{l,m}).
	\end{equation}
	Therefore, we have
	\begin{align*}
		\norm{g}_{L^{p'}(\ell^{r'})}^{p'} &\leq C \sum_{k \in \Z} 2^{kp} \mu( \ell^{r'}(g) > c 2^{k(p-1)}) \\
		&\leq C \sum_{k \in \Z} 2^{kp} \sum_{(l,m) \colon l\geq k} \sigma(E_{l,m}) \\
		&\leq C \sum_{l \in \Z} 2^{lp} \sum_{m \in \N_l} \sigma(E_{l,m}) \\
		&\leq C \sum_{l \in \Z} 2^{lp} \mu(\ell^{r}(f) > 2^{l-1}) \\
		&\leq C \norm{f}_{L^p(\ell^r)}^p,
	\end{align*}
	where we used \eqref{eq:upper_bound_level_set} in the second inequality, Fubini and the bounds on the geometric series in the third, and \eqref{eq:optimal_covering_upperspace_dyadic} in the fourth.
	
	{\textbf{Case III: $p=\infty$ and $1 \leq r < \infty$.}} By construction we have
	\begin{equation*}
		\norm{g}_{L^\infty(\ell^{r'})} \leq \norm{f}_{L^\infty(\ell^r)}^{r-1},
	\end{equation*}
	therefore, by outer H\"{o}lder's inequality, Proposition \ref{thm:Holder_inequality}, we have
	\begin{equation*}
		\norm{g}_{L^1(\ell^{r'})} \leq \norm{g}_{L^\infty(\ell^{r'})} \norm{1_E}_{L^1(\ell^\infty)} \leq \norm{f}_{L^\infty(\ell^r)}^{r-1} \sigma(E).
	\end{equation*}
	
	{\textbf{Case IV: $p=r = \infty$.}} In an analogous way, we have
	\begin{equation*}
		\norm{g}_{L^1(\ell^1)} \leq \sigma(E),
	\end{equation*}
	since by construction $\norm{g}_{L^\infty(\ell^1)} = 1$.
\end{proof}
\begin{remark} \label{rmk:gain_of_the_endpoint}
	Without the crucial property of the decomposition established by \eqref{eq:sparse_collection_upperspace_dyadic}, the argument in \eqref{eq:size_norm_realizing_function_upperspace_dyadic} above produce the empty upper bound 
	\begin{equation*}
		(\ell^{r'}(g)(E))^{r'} \leq \sum_{k \in \Z} 1.
	\end{equation*}
	Nevertheless, when $1 < p < \infty$, in \eqref{eq:level_set_norm_realizing_function_upperspace_dyadic_r} and in \eqref{eq:level_set_norm_realizing_function_upperspace_dyadic_infty} we can already get a summable decay in $l<k$ for the upper bound on the $\ell^{r'}$ size of $g$ over the sets $A \cap (F_l \setminus F_{l+1})$, and it is not necessary to invoke \eqref{eq:sparse_collection_upperspace_dyadic}.
\end{remark}

We have the following lower bounds on the $L^1(X,\omega)$ norm of $fg$, where as above $g$ is defined according to the four $(p,r)$-dependent cases.
\begin{lemma} \label{thm:lower_bound_upperspace_dyadic}
	{\textbf{Case I: $1 \leq p < \infty$ and $1 \leq r \leq \infty$.}} We have
	\begin{equation*}
		\norm{fg}_{L^{1}(X,\omega)} \geq C \norm{f}_{L^p(\ell^r)}^p.
	\end{equation*}
	{\textbf{Case II: $p= \infty$ and $1 \leq r < \infty$.}} We have
	\begin{equation*}
		\norm{fg}_{L^{1}(X,\omega)} \geq C \norm{f}^r_{L^\infty(\ell^r)} \sigma(E).
	\end{equation*}
	{\textbf{Case III: $p= r=\infty$.}} We have
	\begin{equation*}
		\norm{fg}_{L^{1}(X,\omega)} \geq C \norm{f}_{L^\infty(\ell^r)} \sigma(E).
	\end{equation*}
\end{lemma}
\begin{proof}
	{\textbf{Case I: $1 \leq p < \infty$ and $1 \leq r \leq \infty$.}} Let $1 \leq r<\infty$. For every fixed $(k,n)$ such that $E_{k,n}$ is not empty, we have
	\begin{equation} \label{eq:single_cell_estimate}
		\ell^{1}(fg 1_{F_{k,n-1}^c})(E_{k,n}) = 2^{k(p-r)} (\ell^r(f 1_{F_{k,n-1}^c})(E_{k,n}))^r > 2^{kp}, 
	\end{equation}
	where we used \eqref{eq:superlevel_upperspace_dyadic} in the inequality.
	
	For $r=\infty$, by the definition of $g$, we have the same inequality.
	
	Therefore, for $1 \leq r \leq \infty$, we have
	\begin{align*}
		\norm{fg}_{L^1(X, \omega)} & \geq \sum_{k \in \Z} \sum_{n \in \N_k} \norm{fg}_{L^1(E_{k,n} \setminus F_{k,n-1}, \omega)} \\
		& \geq \sum_{k \in \Z} \sum_{n \in \N_k} 2^{kp} \sigma(E_{k,n}) \\
		& \geq C \sum_{l \in \Z} 2^{lp} \sum_{(k,n) \colon l \leq k} \sigma(E_{k,n}) \\
		& \geq C \sum_{l \in \Z} 2^{lp} \mu(\ell^r(f) > 2^{l}) \\
		& \geq C \norm{f}_{L^p(\ell^r)}^p,
	\end{align*}
	where we used \eqref{eq:single_cell_estimate} in the second inequality, the bounds on the geometric series and Fubini in the third, and \eqref{eq:covering_upperspace_dyadic} in the fourth.
	
	{\textbf{Case II: $p= \infty$ and $1 \leq r < \infty$.}} Let $E \in \mathcal{D}$ be the dyadic box associated with $g$, in particular
	\begin{equation*}
		\ell^r(f)(E) \geq C \norm{f}_{L^\infty(\ell^r)}.
	\end{equation*}
	Therefore, we have
	\begin{align*}
		\norm{fg}_{L^1(X,\omega)} &= \norm{f 1_{E}}_{L^r(X,\omega)}^r \\
		&= \ell^r(f)(E)^r \sigma(E)  \\
		&\geq C \norm{f}_{L^\infty(\ell^r)}^r \sigma(E). 
	\end{align*}
	
	{\textbf{Case III: $p= r =\infty$.}} In an analogous way, we have
	\begin{equation*}
		\norm{fg}_{L^1(X,\omega)} \geq C \norm{f}_{L^\infty(\ell^r)} \sigma(E). 
	\end{equation*}
\end{proof}

\begin{proof}[Proof of Theorem \ref{thm:collapsing_Holder_triangular_upper_half_space}, properties $(ii),(iii)$]
	The first inequality in $(ii)$ is given by outer H\"{o}lder's inequality, Proposition \ref{thm:Holder_inequality}.
	
	The second inequality in $(ii)$ is a corollary of the previous Lemmata for $f \in L^p(\ell^r) \cap L^\infty(\ell^r)$. A standard approximation argument yields the case of an arbitrary $f \in L^p(\ell^r)$.
	
	The statement in $(iii)$ is a corollary of the triangle inequality for the $L^1(X,  \omega)$ norm and property $(ii)$.
\end{proof}
We conclude the part of the section about the upper half space with the following observation. 

Let $X$ be the upper half space and $\nu$ the outer measure generated by the pre-measure $\sigma$ on $\mathcal{E}$, the collection of all the open boxes in the upper half space, as in \eqref{eq:outer_measure_from_pre_measure}. In particular, 
\begin{equation} \label{eq:upperspace_continuous_setting}
	\begin{split}
		X &= \R^d \times (0,\infty), \\
		\mathcal{E} &= \{ (x,0) + (0,s)^{d+1} \colon x \in \R^d, s \in (0,\infty) \}, \\
		\sigma(E) &= \abs{B(E)}, \ \ \ \ \text{for every $E \in \mathcal{E}$}, \\
		\omega(y,t) &= t^{-1},
	\end{split}
\end{equation}
where $B(E)$ is the base in $\R^d$ of the box $E$. We observe that $\mathcal{D} \subseteq \mathcal{E}$, and every box in $\mathcal{E}$ can be covered up to a set of measure zero by finitely many dyadic boxes in $\mathcal{D}$ of comparable pre-measure. Therefore, the outer $L^p(\ell^r)$ space quasi-norms in the settings \eqref{eq:upperspace_dyadic_setting} and \eqref{eq:upperspace_continuous_setting} are equivalent by Proposition \ref{thm:equivalence_outer_spaces}. As a consequence, all the previous results obtained in the setting \eqref{eq:upperspace_dyadic_setting} extend to the setting \eqref{eq:upperspace_continuous_setting}. An analogous argument applies to the outer measure structure generated by triangular tents in place of cubic boxes.

We turn now to the finite setting.
\begin{proof} [Proof of Theorem \ref{thm:collapsing_Holder_triangular_finite}]
	The proof of property $(i)$ and, for $1 < p \leq \infty, 1 \leq r < \infty$, of property $(ii)$ follows by arguments analogous to those in the previous proofs, using the decomposition in Proposition \ref{thm:atomic_decomposition_finite}. 
	
	For $p=r \in \{1,  \infty \}$, the statement in $(ii)$ follows by the equivalence between $L^p(\ell^p)$ and $L^p(X,\omega)$ by property $(i)$.
	
	The statement in $(iii)$ is again a corollary of the triangle inequality for the $L^1(X, \omega)$ norm and property $(ii)$.
\end{proof}

\begin{lemma} \label{thm:no_uniformity_finite}
	Let $1 < r \leq \infty$. For every $M > 0$, there exist a finite set $X$, an outer measure $\mu$, a strictly positive weight $\omega$, functions $f, f_n \in L^1(\ell^r)$ such that
	\begin{align*}
		\norm{f}_{L^1(\ell^r)} &\geq M \sup_{\norm{g}_{L^{\infty}(\ell^{r'})} = 1} \norm{fg}_{L^1(\ell^1)}, \\
		\norm{\sum_{n \in \N} f_n}_{L^1(\ell^r)} &\geq M \sum_{n \in \N} \norm{f_n}_{L^1(\ell^r)}.
	\end{align*}
\end{lemma}
\begin{proof}
	Let $\mathcal{D}$ be the set of dyadic intervals. For every $m \in \N$, let 
	\begin{align*}
		X_m &= \{ I \in \mathcal{D} \colon I \subseteq [0,1], \abs{I} \geq 2^{-m} \}, \\
		\mathcal{E}_m &= \{ E_I = \{ J \in \mathcal{D} \colon I \subseteq J \subseteq [0,1] \} \colon I \in X_m, \abs{I} = 2^{-m} \}, \\
		\sigma_m(E_I) &=1, \ \ \ \ \text{for every $I \in X_m, \abs{I} = 2^{-m}$,} \\
		\omega_m(J) &=1, \ \ \ \ \text{for every $J \in X_m$,} \\
		f_m(J) &= 2^m \abs{J}, \\
		f_{I}(J) &= 1_{E_{I}}(J), \ \ \ \ \text{for every $I \in X_m, \abs{I} = 2^{-m}$.}
	\end{align*}
	We have
	\begin{align*}
		\norm{\sum_{I \in X_m, \abs{I} = 2^{-m}} f_I}_{L^1(\ell^r)} &= \norm{f_m}_{L^1(\ell^r)}\geq 2^m \frac{m+1}{2}, \\
		\sum_{I \in X_m, \abs{I} \leq 2^{-m}} \norm{f_I}_{L^1(\ell^r)} &= \sum_{I \in X_m, \abs{I} = 2^{-m}} (m+1)^{\frac{1}{r}} = 2^m (m+1)^{\frac{1}{r}}.
	\end{align*}
	For $m$ big enough, we get the second statement. In particular, this yields a counterexample to the uniformity of the constant in the statement of Theorem \ref{thm:collapsing_Holder_triangular_finite}, property $(iii)$. Therefore, also the uniformity of the constant in the statement of Theorem \ref{thm:collapsing_Holder_triangular_finite}, property $(ii)$ does not hold.
\end{proof}

\section{Equivalence with tent spaces}
In this section we prove the equivalence between the outer $L^p(\ell^r)$ spaces in the upper half space setting \eqref{eq:upperspace_continuous_setting} and the tent spaces $T^p_r$ stated in Theorem \ref{thm:equivalence_tent_outer_spaces}. First, in Lemma \ref{thm:tent_to_outer} we prove the equivalence for certain exponents $p,r$. After that, we extend it to the full range $0 < p,r \leq \infty$ via the K\"{o}the duality result for the outer $L^p(\ell^r)$ spaces, equivalent to that stated in Theorem \ref{thm:collapsing_Holder_triangular_upper_half_space}, property $(ii)$, and the analogous result for tent spaces $T^p_r$, stated in Proposition \ref{thm:tent_duality}.

\begin{lemma} \label{thm:tent_to_outer}
	For $p=\infty, 0<r <\infty$ or $0<p<\infty, r = \infty$, there exists a constant $C=C(p,r)$ such that, for every $f \in L^p(\ell^r)$,
	\begin{equation*}
		\frac{1}{C} \norm{f}_{T^p_r} \leq \norm{f}_{L^p(\ell^r)} \leq C  \norm{f}_{T^p_r}. 
	\end{equation*}
\end{lemma}
\begin{proof}
	Without loss of generality, it is enough to consider the cases
	\begin{equation} \label{eq:restriction_of_cases_1}
		\begin{gathered}
			p= \infty, r= 1, \\
			p=1, r= \infty.
		\end{gathered}
	\end{equation}
	In fact, let $q<\infty$ be the minimum of $p$ and $r$. We have
	\begin{align*}
		\norm{f}_{T^p_r}^q &= \norm{f^{q}}_{T^{p/q}_{r/q}}, \\
		\norm{f}_{L^p(\ell^r)}^q &= \norm{f^{q}}_{L^{p/q}(\ell^{r/q})},
	\end{align*}
	where $\infty/q=\infty$, thus recovering one of the cases in \eqref{eq:restriction_of_cases_1}.
	
	{\textbf{Case I: $p=\infty, r=1$.}} The quantities associated with the spaces $L^\infty(\ell^1),T^\infty_1$ are equivalent by definition, up to a constant determined by a simple covering argument between boxes and tents.
	
	{\textbf{Case II: $p=1, r = \infty$.}} Let $f \in L^1(\ell^\infty)$. For every $\lambda > 0$, let $\mathcal{E}_\lambda \subseteq \mathcal{E}$ be a covering witnessing the super level measure at level $\lambda$ up to a factor $2$. In particular, we have
	\begin{equation*}
		2 \mu(\ell^\infty(f)>\lambda) \geq  \sum_{E \in \mathcal{E}_\lambda} \sigma(E).
	\end{equation*}
	For
	\begin{equation*}
		B_\lambda = \bigcup_{\mathcal{E}_\lambda} 10 B(E) \subseteq \R^d,
	\end{equation*}
	where $10B$ is the cube in $\R^d$ with the same centre of $B$ and $10$ times its side length, we have
	\begin{equation*}
		\abs{B_\lambda} \leq C \sum_{E \in \mathcal{E}_\lambda} \sigma(E) \leq C \mu(\ell^\infty(f)>\lambda).
	\end{equation*}
	Moreover, for every $x \in B_\lambda^c$, we have
	\begin{equation*}
		A_\infty (f) (x) \leq \lambda,
	\end{equation*}
	otherwise we get a contradiction with the definition of $\mathcal{E}_\lambda$. Therefore, we have
	\begin{equation} \label{eq:first_inequality}
		\abs{ \{ x \in \R^d \colon A_\infty (f)(x) > \lambda \} } \leq C \mu(\ell^\infty(f) > \lambda).
	\end{equation}
	
	Now let $f \in T^p_\infty$. For every $\lambda > 0$, let $ D_\lambda$ be
	\begin{equation*}
		D_\lambda = \{ x \in \R^d \colon A_\infty (f)(x) > \lambda \},
	\end{equation*}
	and define
	\begin{equation*}
		E_\lambda = \bigcup_{i \in I_\lambda} 10 Q_i \subseteq X,
	\end{equation*}
	where $\{B(Q_i)\}$ is a Whitney decomposition of $D_\lambda$, and $10 Q$ is the box whose base $B(10Q)$ has the same centre of $B(Q)$ and $10$ times its side length. In particular, we have
	\begin{equation*}
		\mu(E_\lambda) \leq C \abs{D_\lambda}.
	\end{equation*}
	Moreover, for every $E \in \mathcal{E}$, we have
	\begin{equation*}
		\ell^\infty(f 1_{E_\lambda^c})(E) \leq \lambda,
	\end{equation*}
	otherwise we get a contradiction with the definition of $D_\lambda$. Therefore, we have
	\begin{equation} \label{eq:second_inequality}
		\mu(\ell^\infty(f) > \lambda) \leq C \abs{ \{ x \in \R^d \colon A_\infty (f)(x) > \lambda \} }.
	\end{equation}
	
	The desired equivalence follows by integrating the inequalities \eqref{eq:first_inequality}, \eqref{eq:second_inequality} over all levels $\lambda >0$. 
\end{proof}

For the tent spaces $T^p_r$ we have the following K\"{o}the duality result, see for example Theorem 5.2 in \cite{MR3473649}.
\begin{proposition} \label{thm:tent_duality}
	For $1 \leq p , r \leq \infty$, for every $f \in T^p_r$,
	\begin{equation*}
		\sup_{\norm{g}_{T^{p'}_{r'}} = 1} \norm{fg}_{L^1(X, \omega)} \leq \norm{f}_{T^p_r} \leq \sup_{\norm{g}_{T^{p'}_{r'}} = 1} \norm{fg}_{L^1(X, \omega)}. 
	\end{equation*}
\end{proposition}

\begin{proof} [Proof of Theorem \ref{thm:equivalence_tent_outer_spaces}]
	Without loss of generality, it is enough to consider the cases
	\begin{equation*} \label{eq:restriction_of_cases_2}
		\begin{gathered}
			p=r=\infty, \\
			1 < p \leq \infty, r= 1, \\
			p=1, 1 \leq r \leq \infty,
		\end{gathered}
	\end{equation*}
	due to an argument analogous to that in the previous proof.
	
	{\textbf{Case I: $p=r=\infty$.}} The equivalence between $L^\infty(\ell^\infty), T^\infty_\infty$ follows by definition.
	
	{\textbf{Case II: $1 < p \leq \infty, r= 1$.}} For $p=\infty$ the quantities associated with the spaces $L^\infty(\ell^1),T^\infty_1$ are equivalent by Lemma \ref{thm:tent_to_outer}.
	
	For $1 < p < \infty$, let $f \in L^p(\ell^1)$. By Theorem \ref{thm:collapsing_Holder_triangular_upper_half_space}, property $(ii)$, we have
	\begin{equation*}
		\frac{1}{C} \sup_{\norm{g}_{L^{p'}(\ell^\infty)} \leq 1} \norm{fg}_{L^1(X,  \omega)} \leq \norm{f}_{L^p(\ell^1)} \leq C \sup_{\norm{g}_{L^{p'}(\ell^\infty)} \leq 1} \norm{fg}_{L^1(X,  \omega)}.
	\end{equation*}
	Applying Lemma \ref{thm:tent_to_outer} to $g$, we have
	\begin{equation*}
		\frac{1}{C} \sup_{\norm{g}_{T^{p'}_\infty} \leq 1} \norm{fg}_{L^1(X,  \omega)} \leq \norm{f}_{L^p(\ell^1)} \leq C \sup_{\norm{g}_{T^{p'}_\infty} \leq 1} \norm{fg}_{L^1(X,  \omega)}.
	\end{equation*}
	Finally, by Proposition \ref{thm:tent_duality}, we conclude
	\begin{equation*}
		\frac{1}{C} \norm{f}_{T^p_1} \leq \norm{f}_{L^p(\ell^1)} \leq C \norm{f}_{T^p_1}.
	\end{equation*} 
	
	{\textbf{Case III: $p=1,1 \leq r \leq \infty$.}} For $p=1, r=\infty$, the quantities associated with the spaces $L^1(\ell^\infty),T^1_\infty$ are equivalent by Lemma \ref{thm:tent_to_outer}.
	
	For $p=1, 1 \leq r <\infty$, an argument analogous to that used to prove Case II yields the desired equivalence. If $p=r=1$, we use Case I in place of Lemma \ref{thm:tent_to_outer}.
	
	To conclude, we observe that the set of bounded functions with compact support in $X$ is dense in $T^p_r$ for $1 \leq p < \infty, r=1$ and $p=1, 1 \leq r < \infty$. However, these functions are also in $L^p(\ell^r)$. Therefore, the two spaces coincide.
\end{proof}

\section{Hardy-Littlewood-Sobolev inclusions for tent spaces}
In this section we improve over a result of Amenta on continuous inclusions between tent spaces $T^p_r$, see Theorem 2.19 and Lemma 2.20 in \cite{MR3750310}. In his notation, we have the weighted tent spaces $T^{p,r}_s$ defined, for $0<p,r \leq \infty, s \in \R$, by
\begin{equation*}
	T^{p,r}_s = \{ f \colon t^{-ds} f \in T^p_r \}, \ \ \ \ \norm{f}_{T^{p,r}_s}= \norm{t^{-ds} f}_{T^{p}_{r}},
\end{equation*}
where $T^p_r$ is defined in \eqref{eq:tent_space_1} and \eqref{eq:tent_space_2}, and the continuous inclusions
\begin{equation*}
	T^{p,r}_0 \hookrightarrow T^{q,r}_{\frac{1}{q}-\frac{1}{p}}, f \mapsto f,
\end{equation*}
for $0<p<q \leq \infty, 0 < r \leq \infty$. The improvement consists of allowing for two different values of $r$, under certain conditions, in each of the two spaces in the last display.

Due to the equivalence proved in the previous section, we get an analogous result for the outer $L^p(\ell^r)$ spaces in the upper half space setting \eqref{eq:upperspace_continuous_setting}. This result is auxiliary in proving strong type estimates in the following section.
\begin{theorem} \label{thm:inclusion_tent_spaces}
	For $0 < p < q \leq \infty, 0 < r_2 \leq r_1 \leq \infty$, there exists a constant $C=C(p,q,r_1,r_2)$ such that, for every $f \in T^p_{r_1}$,
	\begin{equation*}
		\norm{t^{\frac{d}{p}-\frac{d}{q}} f}_{T^q_{r_2}} \leq C \norm{f}_{T^p_{r_1}}.
	\end{equation*}
	Moreover, for every $f \in L^p(\ell^{r_1})$,
	\begin{equation*}
		\norm{t^{\frac{d}{p}-\frac{d}{q}} f}_{L^q(\ell^{r_2})} \leq C \norm{f}_{L^p(\ell^{r_1})}.
	\end{equation*}
\end{theorem}
The main ingredient is the following. We define a function $a$ to be a $T^p_{r}$-atom associated with the ball $B \subseteq \R^d$ if $a$ is essentially supported in $T(B)$ and
\begin{equation} \label{eq:atomsizecondition}
	\norm{a}_{T^{r}_{r}} \leq \abs{B}^{\frac{1}{r}-\frac{1}{p}}.
\end{equation}

\begin{lemma} \label{thm:atomic_tent_spaces}
	Let $1 < q \leq r_2 \leq r_1 \leq \infty$. Suppose that $a$ is a $T^1_{r_1}$-atom. Then $a$ is in $T^q_{r_2}$ with norm smaller than 1.
\end{lemma}
\begin{proof}
	For $q<\infty$, let $0 < r,s \leq \infty$ be such that 
	\begin{equation*}
		\frac{1}{r}+\frac{1}{r_1}=\frac{1}{r_2}, \ \ \ \ \frac{1}{s}+\frac{1}{r_1}=\frac{1}{q}.
	\end{equation*} 
	We have
	\begin{align*}
		\norm{t^{d-\frac{d}{q}} a}_{T^q_{r_2}} &= \norm{A_{r_2}(t^{d-\frac{d}{q}} a)}_{L^q(B)} \\
		& \leq \norm{A_{r}(t^{d-\frac{d}{q}} 1_{T(B)}) A_{r_1}(a)}_{L^q(B)} \\
		& \leq \norm{A_{r}(t^{d-\frac{d}{q}} 1_{T(B)})}_{L^s(B)} \norm{ A_{r_1}(a)}_{L^{r_1}(B)} \\
		& \leq \abs{B}^{1-\frac{1}{r_1}} \norm{a}_{T^{r_1}_{r_1}} \\
		& \leq 1,
	\end{align*}
	where we used H\"{o}lder's inequality in the first and in the second inequality, and \eqref{eq:atomsizecondition} in the fourth.
	
	For $q=r_2=r_1=\infty$, the statement follows directly from \eqref{eq:atomsizecondition}.
\end{proof}

\begin{proof} [Proof of Theorem \ref{thm:inclusion_tent_spaces}]
	The proof of the first statement follows along the lines of that of Theorem 2.19 in \cite{MR3750310}, using Lemma \eqref{thm:atomic_tent_spaces} above in place of Lemma 2.20.
	
	The second statement then follows by Theorem \ref{thm:equivalence_tent_outer_spaces}.
\end{proof}

\section{Embedding into outer $L^p(\ell^r)$ spaces with a fractional scale factor}
In this section we state and prove a full classification of all positive and negative results regarding strong and weak type estimates for a family of embedding maps with a fractional scale factor from classical $L^p$ spaces on $\R^d$ to outer $L^p(\ell^r)$ spaces in the upper half space setting.

The positive results for $d=1, 1 \leq p=q \leq \infty,r=\infty$ were already proved in \cite{MR3312633}, see Theorem 4.1. Although there $\phi$ was assumed to be smooth and compactly supported, the same argument can be extended with minor adjustments to the test functions satisfying the boundedness and decay condition \eqref{eq:decay_condition} and to all dimensions. 

We conclude the section by stating and proving an embedding theorem with a fractional scale factor for functions in the Hardy space $H^1(\R^d)$ into the outer $L^1(\ell^\infty)$ space. The embedded function in this case is that defined in \eqref{eq:embedding_map} for a smooth test function $\phi \in \mathcal{S}(\R^d)$.

\begin{theorem} \label{thm:embedding}
	Let 
	\begin{equation} \label{eq:zerocase}
		1  \leq p,q \leq \infty,  0< r \leq \infty.
	\end{equation}
	Then, for $(p,q,r)$ satisfying one of the following conditions, which are also displayed in Fig. 1 below,
	\begin{equation}  \label{eq:case}
		\begin{gathered}
			1 < p < q \leq \infty, 0 < r \leq \infty, \\
			1 < p = q \leq \infty, r = \infty, \\
			p = 1, q = \infty, 0<r\leq \infty, 
		\end{gathered}
	\end{equation}
	there exists a constant $C=C(p,q,r,d,\varepsilon)$ such that, for every $f \in L^p(\R^d)$,
	\begin{equation*} \label{eq:strong_embedding}
		\norm{t^{\frac{d}{p}-\frac{d}{q}} F(f)}_{L^{q}(\ell^r)} \leq C \norm{f}_{L^p(\R^d)}.
	\end{equation*}
	For all the triples $(p,q,r)$ satisfying \eqref{eq:zerocase} but none of the conditions in \eqref{eq:case}, no strong type $(p,q)$ estimate holds.
	
	Moreover, for $(p,q,r)$ satisfying one of the following conditions, which are also displayed in Fig. 1 below,
	\begin{equation} \label{eq:case_weak}
		\begin{gathered}
			1=p < q < \infty, 0 < r \leq \infty, \\
			p= q = 1, r = \infty, 
		\end{gathered}
	\end{equation}
	there exists a constant $C=C(q,r,d,\varepsilon)$ such that, for every $f \in L^1(\R^d)$,
	\begin{equation*} \label{eq:weak_embedding}
		\norm{ F(f)}_{L^{q, \infty}(\ell^r)} \leq C \norm{f}_{L^1(\R^d)}.
	\end{equation*}
	For all the triples $(p,q,r)$ satisfying \eqref{eq:zerocase} but none of the conditions in \eqref{eq:case},\eqref{eq:case_weak}, no weak type $(p,q)$ estimate holds.
\end{theorem}

\begin{center}
	\begin{tikzpicture}[domain=0:0.1,xscale=2,yscale=2]
	\draw[->] (-0.5,0) -- (1.5,0) node [above,right]{$\frac{1}{p}$};
	\draw[->] (0,-0.5) -- (0,1.5) node [above,left]{$\frac{1}{q}$};
	\filldraw[draw=blue,ultra thick, fill=blue!20] (0,0) -- (1,0) -- (1,1) -- cycle;
	\node[draw,align=left] at (1.5,1.5) {$r=\infty$ \\ weak type $(p,q)$};
	\node[blue] at (0,0) {\textbullet};
	\node[blue] at (1,0) {\textbullet};
	\node[blue] at (1,1) {\textbullet};
	\end{tikzpicture} 
	\begin{tikzpicture}[domain=0:0.1,xscale=2,yscale=2]
	\draw[->] (-0.5,0) -- (1.5,0) node [above,right]{$\frac{1}{p}$};
	\draw[->] (0,-0.5) -- (0,1.5) node [above,left]{$\frac{1}{q}$};
	\filldraw[draw=blue,ultra thick, fill=blue!20] (0,0) -- (1,0) -- (1,1) -- cycle;
	\node[blue] at (0,0) {\textbullet};
	\node[blue] at (1,1) {\textbullet};
	\node[white] at (1,1) {$\bullet$};
	\draw[white,thick] (1,0) -- (1,1);
	\node[draw,align=left] at (1.5,1.5) {$r=\infty$ \\ strong type $(p,q)$};
	\node[blue] at (1,0) {\textbullet};
	\end{tikzpicture}
	
	\begin{tikzpicture}[domain=0:0.1,xscale=2,yscale=2]
	\draw[->] (-0.5,0) -- (1.5,0) node [above,right]{$\frac{1}{p}$};
	\draw[->] (0,-0.5) -- (0,1.5) node [above,left]{$\frac{1}{q}$};
	\filldraw[draw=blue,ultra thick, fill=blue!20] (0,0) -- (1,0) -- (1,1) -- cycle;
	\node[blue] at (0,0) {\textbullet};
	\node[blue] at (1,1) {\textbullet};
	\draw[white,thick] (0,0) -- (1,1);
	\node[draw,align=left] at (1.5,1.5) {$0<r<\infty$ \\ weak type $(p,q)$};
	\node[blue] at (1,0) {\textbullet};
	\node[white] at (0,0) {$\bullet$};
	\node[white] at (1,1) {$\bullet$};
	\node[blue] at (1,0) {$\bullet$};
	\end{tikzpicture}
	\begin{tikzpicture}[domain=0:0.1,xscale=2,yscale=2]
	\draw[->] (-0.5,0) -- (1.5,0) node [above,right]{$\frac{1}{p}$};
	\draw[->] (0,-0.5) -- (0,1.5) node [above,left]{$\frac{1}{q}$};
	\filldraw[draw=blue,ultra thick, fill=blue!20] (0,0) -- (1,0) -- (1,1) -- cycle;
	\node[blue] at (0,0) {\textbullet};
	\node[blue] at (1,1) {\textbullet};
	\draw[white,thick] (1,0) -- (1,1);
	\draw[white,thick] (0,0) -- (1,1);
	\node[draw, align=left] at (1.5,1.5) {$0<r<\infty$ \\ strong type $(p,q)$};
	\node[blue] at (1,0) {\textbullet};
	\node[white] at (0,0) {$\bullet$};
	\node[white] at (1,1) {$\bullet$};
	\node[blue] at (1,0) {$\bullet$};
	\end{tikzpicture}
	
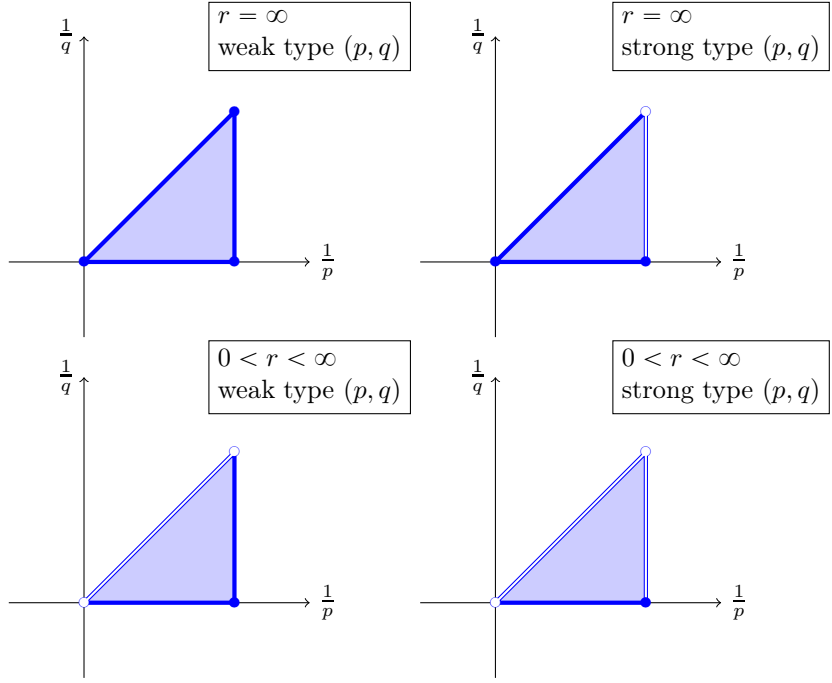
	\textit{Figure 1: range of exponents $p,q,r$ and weak/strong type estimates.}
\end{center}

In the next proof, the constants $c,C$ are allowed to depend on $d,\varepsilon,p,q,r$ but not on $f$.
\begin{proof}[Proof of Theorem \ref{thm:embedding}]
	Without loss of generality, we can assume $f$ to be nonnegative. In fact, by definition \eqref{eq:embedding_map}, we have the pointwise bound
	\begin{equation*}
		\abs{F_\phi(f)(y,t)} \leq F_{\abs{\phi}}(\abs{f})(y,t) \leq F(\abs{f})(y,t).
	\end{equation*}
	In particular, we have
	\begin{equation*}
		F(f)(y,t) = \int_{\R^d} f(z) t^{-d} (1+t^{-1}\abs{y-z})^{-d-\varepsilon} \diff z.
	\end{equation*}
	This expression can be bounded either by means of the centred maximal function
	\begin{equation} \label{eq:maximal_bound}
		F(f)(y,t) \leq C Mf(y),	
	\end{equation}
	or by Young's convolution inequality 
	\begin{equation} \label{eq:convolution_bound}
		F(f)(y,t) \leq C t^{-\frac{d}{p}} \norm{f}_{L^p(\R^d)}.
	\end{equation}
	
	\subsection{Strong type $(p,q)$ estimates for $0 < r \leq \infty$ in the range for $p \neq 1,q$ displayed in Fig. 1}
	The strong type $(p,q)$ estimates in the range $1 < p < q \leq \infty, 0 < r \leq \infty$ follow by the already known strong type $(p,p)$ estimate for $1 < p \leq \infty, r= \infty$ and Theorem \ref{thm:inclusion_tent_spaces}.
	
	\subsection{Strong type $(1,\infty)$ estimates for $0 < r \leq \infty$}
	We aim to prove that, for every $E \in \mathcal{E}$,
	\begin{equation} \label{eq:claim}
		\ell^r( t^{d} F(f))(E) \leq C \norm{f}_{L^1(\R^d)}.
	\end{equation}
	
	If $r=\infty$, the claim follows by \eqref{eq:convolution_bound}. 
	
	Now let $0<r<\infty$. By Theorem \ref{thm:collapsing_Holder_triangular_upper_half_space}, property $(iii)$, the decay property of $\phi$, and the translation invariance of the $L^\infty(\ell^r)$ quasi-norm, it is enough to prove the inequality assuming that $f$ is supported in $(-1,1)^d$ and $\phi = 1_{(-1,1)^d}$. In this case, we have
	\begin{equation*}
		F_\phi (f) (y,t) \leq C t^{-d} \norm{f 1_{ y + (-t,t)^d  }}_{L^1(\R^d)} 1_{ \{(-1-s,1+s)^d \times \{s\}, s >0 \} } (y,t),
	\end{equation*}
	and it is enough to prove \eqref{eq:claim} for the elements of $\mathcal{E}$ of the form
	\begin{equation*}
		E_{x,u} =( x+ (-u,u)^d) \times (0,2u) \in \mathcal{E},
	\end{equation*}
	for every $u>0, x \in (-1-u,1+u)^d$. We distinguish two cases, $r \geq 1$ and $0 < r < 1$.
	
	{\textbf{Case I: $r \geq 1$.}} Let $r=1$. We have
	\begin{align*}
		\ell^1 (t^{d} F_\phi(f)) (E_{x,u}) &\leq \frac{C}{u^d} \int_0^{2u} \int_{x+ (-u,u)^d} \int_{(-1,1)^d} f(z) 1_{ y + (-t,t)^d  } (z) \diff z \diff y \frac{\diff t}{t} \\ 
		&\leq \frac{C}{u^d} \int_{(-1,1)^d} f(z) \int_0^{2u} \int_{x+ (-u,u)^d} 1_{ z + (-t,t)^d  } (y) \diff y \frac{\diff t}{t} \diff z \\
		&\leq \frac{C}{u^d}\norm{f}_{L^1(\R^d)} \int_0^{2u} t^d \frac{\diff t}{t} \\
		&\leq C \norm{f}_{L^1(\R^d)},
	\end{align*}
	where we used Fubini in the second inequality.
	
	If $1 < r < \infty$, Proposition \ref{thm:log_convexity} implies the strong type $(1,\infty)$ estimate for $L^\infty(\ell^r)$ from those for $L^\infty(\ell^1),L^\infty(\ell^\infty)$.
	
	{\textbf{Case II: $0<r<1$.}} We have
	\begin{align*}
		\ell^r (t^{d} F_\phi(f)) (E_{x,u}) &\leq \ell^1 (t^{d-\frac{1}{2}} F_\phi(f)) (E_{x,u}) \ell^{\frac{r}{1-r}} (t^{\frac{1}{2}}) (E_{x,u}) \\ 
		&\leq C (\frac{1}{u^d} \int_0^{2u} \int_{x+ (-u,u)^d}  t^{-\frac{1}{2}} \int_{(-1,1)^d} f(z) 1_{ y + (-t,t)^d  } (z) \diff z \diff y \frac{\diff t}{t}) \times \\
		& \ \ \ \ \times (\frac{1}{u^d} \int_0^{2u} \int_{x+ (-u,u)^d}  t^{\frac{r}{2(1-r)}} \diff y \frac{\diff t}{t})^{\frac{1-r}{r}} \\
		&\leq C \norm{f}_{L^1(\R^d)} (\frac{1}{u^d} \int_0^{2u} t^{d-\frac{1}{2}} \frac{\diff t}{t}) (\int_0^{2u} t^{\frac{r}{2(1-r)}} \frac{\diff t}{t})^{\frac{1-r}{r}} \\
		&\leq C \norm{f}_{L^1(\R^d)},
	\end{align*}
	where we used H\"{o}lder's inequality with exponents $(1,\frac{r}{1-r})$ in the first inequality, and then we proceeded as in the previous case.
	
	\subsection{Weak type $(1,q)$ estimates for $0 < r \leq \infty$ in the range for $q\neq \infty$ displayed in Fig. 1}
	
	We aim to prove that, for every $\lambda > 0$,
	\begin{equation*}
		\lambda^q \mu(\ell^r(t^{d-\frac{d}{q}}F(f)) > \lambda) \leq C \norm{f}_{L^1(\R^d)}^q.
	\end{equation*}
	This requires to construct, for every $\lambda > 0$, a set with appropriate outer measure approximating the super level measure at level $\lambda$. 
	
	For fixed $f$ and $\lambda>0$, let $D_\lambda$ be the set
	\begin{equation*}
		D_\lambda = \{ x \in \R^d \colon Mf(x) > \lambda^q \norm{f}_{L^1(\R^d)}^{1-q} \}.
	\end{equation*}
	We have
	\begin{equation*}
		\abs{D_\lambda} \leq C \lambda^{-q} \norm{f}_{L^1(\R^d)}^q,
	\end{equation*}
	because of the weak type $(1,1)$ estimate for the maximal function operator on $\R^d$.
		
	Let $ \{ B_i \colon i \in I_\lambda \} $ be a Whitney covering of $D_\lambda$ up to a set of measure $0$ by pairwise disjoint open dyadic cubes in $\R^d$, and denote by $x_i$ and $s_i$ the centre and the side length of $B_i$, respectively. Let $Q(B_i)=Q_i \in \mathcal{D}$ be the dyadic box over the cube $B_i$, and define
	\begin{equation*}
		E_\lambda = \bigcup_{i \in I_\lambda} Q_i \subseteq X.
	\end{equation*}
	
	In particular, we have
	\begin{equation*}
		\mu(E_\lambda) \leq \abs{D_\lambda} \leq C \lambda^{-q} \norm{f}_{L^1(\R^d)}^q.
	\end{equation*}
	We are left with proving that for every $E \in \mathcal{E}$,
	\begin{equation*}
		\ell^r( t^{d-\frac{d}{q}} F(f) 1_{E_\lambda^c}) (E) \leq C \lambda.
	\end{equation*}
	
	If $(x,s) \in E_\lambda^c, x \in D_\lambda$, then $x \in Q_i$ for some $i \in I_\lambda$, $s > s_i$, and there exists $u \in \mathbb{S}^{d-1}$ such that $x+s' u \in D_\lambda^c$, for $c s_i \leq s' \leq C s_i $. As a consequence, for $ t \geq s $, we have
	\begin{equation*}
		t^{ d- \frac{d}{q} } F(f) (x,t)  \leq C (t+s')^{ d- \frac{d}{q} } F(f) ( x + s' u, t + s' ) .
	\end{equation*}
	Therefore, we have
	\begin{equation*}
		\ell^r( t^{d-\frac{d}{q}} F(f) 1_{E_\lambda^c}) (E) \leq C \sup_{x \in D_\lambda^c} \norm{t^{d-\frac{d}{q}} F(f) }_{L^r(\{ x \} \times (0,\infty), \frac{\diff t}{t})},
	\end{equation*}
	and it is enough to show that for every $x \in D_\lambda^c$, we have
	\begin{equation*}
		\norm{t^{d-\frac{d}{q}} F(f) }_{L^r(\{ x \} \times (0,\infty),\frac{\diff t}{t})} \leq C \lambda.
	\end{equation*}
	
	We split the norm on the left hand side at height $0< R(x)< \infty$ soon to be fixed
	\begin{equation} \label{eq:splitting_norm}
		\norm{t^{d-\frac{d}{q}}  F(f)}_{L^r(\{ x \} \times(0,R(x)),\frac{\diff t}{t})}+\norm{t^{d-\frac{d}{q}} F(f)}_{L^r(\{ x \} \times(R(x),\infty),\frac{\diff t}{t})}.
	\end{equation}
	We bound $F(f)$ by \eqref{eq:maximal_bound} in the first summand obtaining
	\begin{equation*}
		C Mf(x) R(x)^{d-\frac{d}{q}},
	\end{equation*}
	and by \eqref{eq:convolution_bound} in the second summand obtaining
	\begin{equation*}
		C \norm{f}_{L^1(\R^d)} R(x)^{- \frac{d}{q}}.
	\end{equation*}
	If $0<r<\infty$, we require the additional hypothesis $q >1$ to guarantee the $L^r$-integrability at $0$ of the estimate for the first summand.
	
	Optimizing the choice of $R(x)$ with
	\begin{equation*}
		R(x) = C Mf (x)^{-\frac{1}{d}} \norm{f}_{L^1(\R^d)}^{\frac{1}{d}},
	\end{equation*}
	we get the bound for \eqref{eq:splitting_norm}
	\begin{equation*}
		C Mf(x)^{\frac{1}{q}} \norm{f}_{L^1(\R^d)}^{1-\frac{1}{q}}.
	\end{equation*}
	
	We conclude by the estimate for every $x \in D_\lambda^c$,
	\begin{equation*}
		Mf(x) \leq \lambda^q \norm{f}_{L^1(\R^d)}^{1-q}.
	\end{equation*}
		
	\subsection{Counterexample to the strong type $(1,q)$ estimates for $1 \leq q < \infty, 0<r \leq \infty$}
	In the following counterexamples we are going to use test functions $\phi$ satisfying the condition \eqref{eq:decay_condition} with a multiplicative factor different from $1$. While it does not effect the nature of the counterexamples, it spares us the definition of other appropriate constants. 
	
	For $f= 1_{(-1,1)^d}, \phi = 1_{(-1,1)^d}$, we have
	\begin{equation*}
		F_\phi(f)(y,t) \geq t^{-d} 1_{\{ (-s,s)^d \times \{s \} , s \geq 1 \}} (y,t).
	\end{equation*}
	
	For every $u \geq 1$, let 
	\begin{equation*}
		E_u = (0,u)^{d+1} \in \mathcal{E}.
	\end{equation*}
	Then, for $0 < r < \infty$, we have
	\begin{equation*}
		\ell^r (t^{d-\frac{d}{q}} F (f)  1_{(\R^d \times (0,u))^c} ) (E_{2u}) \geq (\frac{1}{(2u)^d} \int_{u}^{2u} \int_{(0,u)^d} t^{-\frac{dr}{q}} \diff y \frac{\diff t}{t})^{\frac{1}{r}} \geq C  u^{-\frac{d}{q}}, 
	\end{equation*}
	and it is easy to see that, for $r=\infty$, we have
	\begin{equation*}
		\ell^{\infty}(t^{d-\frac{d}{q}} F (f)  1_{(\R^d \times (0,u))^c} ) (E_{2u}) = u^{-\frac{d}{q}}. 
	\end{equation*} 
	
	Therefore, for every fixed $u \geq 1$, if $A \subseteq X$ is such that
	\begin{equation*}
		\ell^r (t^{d-\frac{d}{q}} F (f)  1_{A^c} ) (E_{2u}) \leq  C  u^{-\frac{d}{q}}, 
	\end{equation*}
	then $A \setminus (\R^d \times (0,u))  \neq \emptyset$, hence we have
	\begin{equation*}
		\mu(\ell^r(t^{d-\frac{d}{q}} F(f))> C  u^{-\frac{d}{q}}) \geq u^d.
	\end{equation*}
	
	As a consequence, we have
	\begin{equation*}
		\norm{t^{d-\frac{d}{q}} F_\phi (f)}^q_{L^q(\ell^r)} \geq C \int_0^{C} u^{-d} \mu(\ell^r(t^{d-\frac{d}{q}} F(f))> C  u^{-\frac{d}{q}}) \frac{\diff u}{u} = \infty.  
	\end{equation*}
	
	\subsection{Counterexample to the weak type $(p,q)$ estimates for $1 \leq q \leq p \leq \infty,0<r < \infty$ and $1 \leq q < p \leq \infty, r= \infty$}
	
	For $f,\phi$ as above, we have
	\begin{equation*}
		F_\phi(f)(y,t) \geq 1_{\{ (-1+s,1-s)^d \times \{s \} , s \leq 1 \}} (y,t).
	\end{equation*}
	
	For every $x \in (0,\frac{1}{4})^d, u \leq \frac{1}{4}$, let 
	\begin{equation*}
		E_{x,u} =( x+ (-u,u)^d) \times (0,2u) \in \mathcal{E}.
	\end{equation*}
	Then, for $1 \leq q \leq p \leq \infty,0<r < \infty$, we have
	\begin{equation*}
		\ell^r (t^{\frac{d}{p}-\frac{d}{q}} F_\phi (f)) (E_{x,u}) \geq (\frac{1}{(2u)^d} \int_0^{2u} \int_{ x+ (-u,u)^d} t^{\frac{dr}{p}-\frac{dr}{q}} \diff y \frac{\diff t}{t})^{\frac{1}{r}} = \infty, 
	\end{equation*}
	thus exhibiting a counterexample in the case $p=q=\infty$. Moreover, it is easy to see that, for $1 \leq q < p \leq \infty, r= \infty$, we have
	\begin{equation*}
		\ell^\infty (t^{\frac{d}{p}-\frac{d}{q}} F_\phi (f)) (E_{x,u}) = \infty. 
	\end{equation*} 
	
	Let $A \subseteq (-1,1)^d \times (0,\infty)$ be such that, for every $x \in (0,\frac{1}{4})^d, u \leq \frac{1}{4}$,
	\begin{equation} \label{eq:condition_on_A}
		\ell^r (t^{\frac{d}{p}-\frac{d}{q}} F (f)  1_{A^c} ) (E_{x,u}) < \infty.
	\end{equation}
	For every finite collection $\mathcal{E}' \subseteq \mathcal{E}$ covering $A$, let
	\begin{equation*}
		A_{\mathcal{E}'} = \bigcup_{E \in \mathcal{E}'} \overline{B(E)},
	\end{equation*}
	where $B(E)$ is the base in $\R^d$ of $E$, and $\overline{B}$ is the closure of $B$ in $\R^d$. If $A_{\mathcal{E}'} \cap [0,\frac{1}{4}]^d \neq \emptyset$, there would exist $x,u$ such that $E_{x,u} \cap A = \emptyset$, hence contradicting \eqref{eq:condition_on_A}. Therefore, for every $\lambda > 0$, we have
	\begin{equation*}
		\mu(\ell^r(t^{\frac{d}{p}-\frac{d}{q}} F_\phi(f))> \lambda) \geq C,
	\end{equation*}
	where $C$ does not depend on $\lambda$.
	
	As a consequence, for $q \neq \infty$, we have
	\begin{equation*}
		\norm{t^{\frac{d}{p}-\frac{d}{q}} F_\phi (f)}_{L^{q,\infty}(\ell^r)}^q \geq C \sup_{\lambda>0} \lambda^q = \infty.  
	\end{equation*}
\end{proof}

Before stating and proving the embedding result for functions in $H^1(\R^d)$, we recall the definition of $H^1$-atom. A function $f$ is a $H^1$-atom associated with the cube $B \subseteq \R^d$ if $f$ is essentially supported in $B$ and
\begin{equation} \label{eq:H1atomsizecondition}
	\int_B f(x) \diff x = 0, \ \ \ \ \norm{f}_{L^\infty(\R^d)} \leq \abs{B}^{-1}.
\end{equation}
\begin{proposition} \label{thm:embedding_Hardy_space}
	Let $\varphi \in \mathcal{S}(\R^d)$. Then there exists a constant $C=C(d,\varphi)$ such that, for every $f \in H^1(\R^d)$,
	\begin{equation*}
		\norm{F_\varphi(f)}_{L^1(\ell^\infty)} \leq C \norm{f}_{H^1(\R^d)}.
	\end{equation*}
\end{proposition}
\begin{proof}
	By Theorem \ref{thm:collapsing_Holder_triangular_upper_half_space}, property $(iii)$, the decay properties of $\varphi$ and its derivatives, and the definition of the Hardy space $(H^1(\R^d), \norm{\cdot}_{H^1(\R^d)})$, it is enough to prove the inequality assuming that $\varphi$ is a smooth function compactly supported in a cube of side length $2$ and $f$ is a $H^1$-atom associated with a cube $B$. Moreover, due to the translation invariance of the $L^1(\ell^\infty)$ quasi-norm, we can assume that both $f,\varphi$ are supported in cubes centred in the origin. Therefore it is enough to show that
	\begin{equation*}
		\norm{F_\varphi(f)}_{L^1(\ell^\infty)} \leq C.
	\end{equation*}
	
	Let $2B$ be the cube with the same centre of $B$ and double the side length. For $0 < t <\abs{B}^{\frac{1}{d}},  y \in 2B$, we have
	\begin{equation*}
		\abs{F_\varphi (f) (y,t)} \leq C \abs{B}^{-1},	
	\end{equation*}
	where we used the $L^\infty$ bounds for $f$.
	
	For $t \geq \abs{B}^{\frac{1}{d}}, y \in (-\abs{B}^{\frac{1}{d}} - t, \abs{B}^{\frac{1}{d}} + t)^d$, we have
	\begin{align*}
		\abs{F_\varphi (f) (y,t)} & = C t^{-d} \abs{\int_{B} f(z) \varphi(t^{-1}(y-z)) \diff z} \\
		& = C t^{-d} \abs{\int_{B} f(z)( \varphi(t^{-1}(y-z)) - \varphi(t^{-1} y)) \diff z}  \\
		& \leq C t^{-d} \int_{B} \abs{f(z)} t^{-1} \abs{z} \diff z \\
		& \leq C \abs{B}^{\frac{1}{d}} t^{-(d+1)} ,
	\end{align*}
	where we used the $L^\infty$ bounds, the localized support and the cancellation property of $f$ together with the smoothness of $\varphi$.
	
	For all the others $(y,t)$, we have $F_\varphi (f)$ is $0$, since the supports of $f$ and the dilated version of $\varphi$ are disjoint. 
		
	As a consequence, for $\lambda > C \abs{B}^{-1}$, we have
	\begin{equation*}
		\mu(\ell^\infty ( F_\varphi(f))> \lambda) = 0,
	\end{equation*}
	and for $0 < \lambda \leq C \abs{B}^{-1}$, we have
	\begin{equation*}
		\mu(\ell^\infty ( F_\varphi(f))> \lambda) \leq C \abs{B}^{\frac{1}{d+1}} \lambda^{-\frac{d}{d+1}}.
	\end{equation*}
	
	Therefore, we have
	\begin{equation*}
		\norm{F_\varphi(f)}_{L^1(\ell^\infty)} \leq C \int_0^{C \abs{B}^{-1}} \mu(\ell^\infty (F_\varphi(f))> \lambda) \diff \lambda \leq C.
	\end{equation*}
\end{proof}

\section{Applications}
In this section we show some applications of the strong type estimates in Theorem \ref{thm:embedding} and Proposition \ref{thm:embedding_Hardy_space}. We use them to give alternative proofs of the Hardy-Littlewood-Sobolev inequality, and the Gagliardo-Nirenberg-Sobolev inequality up to the endpoint in the spirit of the two-step program outlined in the introduction.
\begin{theorem} [HLS inequality]
	For $1<p,q< \infty, 0< \alpha <d$ such that
	\begin{equation*}
		\frac{1}{p}+\frac{1}{q}+\frac{\alpha}{d}=2,
	\end{equation*}
	there exists a constant $C=C(p,q,d)$ such that, for every $f \in L^p(\R^d), g \in L^q(\R^d)$,
	\begin{equation*}
		\abs{\int_{\R^{2d}} \frac{f(x)g(y)}{\abs{x-y}^\alpha} \diff x \diff y} \leq C \norm{f}_{L^p(\R^d)} \norm{g}_{L^q(\R^d)}.
	\end{equation*}
\end{theorem}

\begin{proof}
	Let $\psi \in \mathcal{S}(\R)$ be such that $ \supp \hat{\psi} \subseteq [\frac{1}{2},2], \int_0^\infty \hat{\psi}^2(t) \frac{\diff t}{t}=1$, and define $\Psi, \Phi \in \mathcal{S}(\R^d)$ by
	\begin{equation*}
		\hat{\Psi}(\xi) = \hat{\psi}(\abs{\xi}), \hat{\Phi}(\xi)=\abs{\xi}^{\alpha-d} \hat{\psi}(\abs{\xi}).
	\end{equation*}
	Let $f,g \in \mathcal{S}(\R^d)$. By a frequency localization argument, we have
	\begin{align*}
		\abs{\int_{\R^{2d}} \frac{f(x)g(y)}{\abs{x-y}^\alpha} \diff x \diff y} 
		&\leq C \abs{\int_{\R^{2d}} \hat{f}(\xi)\hat{g}(\eta) \abs{\xi-\eta}^{\alpha-d}  \delta(\xi+\eta) \diff \xi \diff \eta} \\
		&\leq C \abs{\int_{\R^{2d} \times (0,\infty)} \hat{f}(\xi)\hat{g}(\eta) \abs{\xi-\eta}^{\alpha-d}  \delta(\xi+\eta) \hat{\psi}^2(t) \diff \xi \diff \eta \frac{\diff t}{t}} \\
		&\leq C \abs{\int_{\R^{d} \times (0,\infty)} t^{d-\alpha} \hat{f}(\xi) \hat{\psi}(\abs{\xi} t) \hat{g}(-\xi) (t \abs{\xi})^{\alpha-d}   \hat{\psi}(\abs{\xi} t) \diff \xi \frac{\diff t}{t}} \\
		&\leq C \abs{\int_{\R^{d} \times (0,\infty)} t^{d-\alpha} F_\Psi(f)(y,t) G_\Phi(g)(y,t) \diff y \frac{\diff t}{t}}.
	\end{align*}
	By Theorem \ref{thm:collapsing_Holder_triangular_upper_half_space}, property $(i)$, the integral in the last display is bounded by
	\begin{equation*}
		\norm{t^{d-\alpha}F_\Psi(f) G_\Phi(g) }_{L^1(\ell^1)}.
	\end{equation*}
	Applying outer H\"{o}lder's inequality, Proposition \ref{thm:Holder_inequality}, we estimate it in terms of
	\begin{equation*}
		\norm{t^{d-\alpha} F_\Psi(f)}_{L^{q'}(\ell^1)} \norm{G_\Phi(g)}_{L^{q}(\ell^\infty)},
	\end{equation*}
	which by the strong type estimates in Theorem \ref{thm:embedding} is bounded by
	\begin{equation*}
		\norm{f}_{L^p(\R^d)} \norm{g}_{L^q(\R^d)}.	
	\end{equation*}
	A standard approximation argument yields the result for arbitrary $f \in L^p(\R^d), g \in L^q(\R^d)$.
\end{proof}

\begin{theorem} [GNS inequality]
	For $1 \leq p < d$, there exists a constant $C=C(p,d)$ such that, for every $ f \in W^{1,p}(\R^d)$,
	\begin{equation*}
		\norm{f}_{L^{p_\ast}(\R^d)} \leq C \norm{\nabla f}_{L^p(\R^d)},
	\end{equation*}
	where $p_\ast = \frac{dp}{d-p}$. 
	
	Moreover, there exists a constant $C=C(d)$ such that, for every $ f \in W^{1,d}(\R^d)$,
	\begin{equation*}
		\norm{f}_{BMO(\R^d)} \leq C \norm{\nabla f}_{L^d(\R^d)}.
	\end{equation*}
\end{theorem}

\begin{proof}
	Let $\{ \varphi_i \}_{i=1}^d$ be a smooth partition of the unity on the set $\{ \frac{1}{2} \leq \abs{\xi} \leq 2 \}$ such that $\supp \varphi_i \subseteq \{ \abs{\xi_i} > \frac{1}{4d} \} \cap \{ \frac{1}{4} \leq \abs{\xi} \leq 4 \}$. 
	
	For $\psi \in \mathcal{S}(\R)$ as above, let $\Psi_i \in \mathcal{S}(\R^d)$ be defined by
	\begin{equation*}
		\hat{\Psi}_i(\xi) = \frac{\hat{\psi}(\abs{\xi})}{\xi_i} \varphi_{i}(\xi).
	\end{equation*}
	For $1 < p < d$, let $f,g \in \mathcal{S}(\R^d)$. By a frequency localization argument, we have
	\begin{align*}
		\abs{\< f,g \>} &\leq C \abs{\int_{\R^{2d}} \hat{f}(\xi) \hat{g}(\eta) \delta(\xi+\eta) \diff \xi \diff \eta} \\
		&\leq C \abs{\int_{\R^{2d} \times (0,\infty)} \hat{f}(\xi)\hat{g}(\eta) \delta(\xi+\eta) \hat{\psi}^2(t) \diff \xi \diff \eta \frac{\diff t}{t}} \\
		&\leq C \sum_{i=1}^d \abs{\int_{\R^{d} \times (0,\infty)} t \xi_i \hat{f}(\xi) \frac{\hat{\psi}(\abs{\xi} t)}{t \xi_i}  \varphi_i (\xi) \hat{g}(-\xi) \hat{\psi}(\abs{\xi} t)  \diff \xi \frac{\diff t}{t}} \\
		&\leq C \sum_{i=1}^d \abs{\int_{\R^d \times (0,\infty)} t F_{\Psi_i} (\partial_i f) (y,t) G_\Psi(g)(y,t) \diff y \frac{\diff{t}}{t}}.
	\end{align*}
	By Theorem \ref{thm:collapsing_Holder_triangular_upper_half_space}, property $(i)$, the integral in the last display is bounded by
	\begin{equation*}
		\sum_{i=1}^d \norm{t F_{\Psi_i}(\partial_i f) G_\Psi(g)}_{L^1(\ell^1)}.
	\end{equation*}
	Applying outer H\"{o}lder's inequality, Proposition \ref{thm:Holder_inequality}, we estimate it in terms of
	\begin{equation*}
		\sum_{i=1}^d \norm{t F_{\Psi_i}(\partial_i f)}_{L^{p_\ast}(\ell^1)} \norm{G_\Psi(g)}_{L^{{p_\ast}'}(\ell^\infty)},
	\end{equation*}
	which by the strong type estimates in Theorem \ref{thm:embedding} is bounded by
	\begin{equation*}
		\sum_{i=1}^d \norm{\partial_i f}_{L^p(\R^d)} \norm{g}_{L^{{p_\ast}'}(\R^d)}.	
	\end{equation*}
	The duality between $L^p(\R^d)$ spaces and the density of Schwartz functions in $L^p(\R^d)$ yield the desired inequality.	A standard approximation argument yields the result for arbitrary $f \in W^{1,p}(\R^d)$.
	
	For $p=d$, we proceed in the same way with $f \in \mathcal{S}(\R^d)$ and $ g \in H^1(\R^d) \cap \mathcal{S}(\R^d)$, getting
	\begin{equation*}
		\abs{\< f,g \>} \leq \sum_{i=1}^d \norm{t F_{\Psi_i}(\partial_i f)}_{L^{\infty}(\ell^1)} \norm{G_\Psi(g)}_{L^{1}(\ell^\infty)},
	\end{equation*}
	which by the strong type estimates in Theorem \ref{thm:embedding} and by Proposition \ref{thm:embedding_Hardy_space} is bounded by
	\begin{equation*}
		\sum_{i=1}^d \norm{\partial_i f}_{L^d(\R^d)} \norm{g}_{H^1(\R^d)}.	
	\end{equation*}
	The duality between the spaces $BMO(\R^d)$ and $H^1(\R^d)$ and the density of Schwartz functions in $H^1(\R^d)$ yield the desired inequality.
	A standard approximation argument yields the result for arbitrary $f \in W^{1,d}(\R^d)$.
	
	For $p=1,d > 1$, the statement can be classically proved by the Loomis-Whitney inequality.
\end{proof}

\section{Appendix: outer $L^p$ space theory}
In this Appendix we review the theory of outer $L^p$ spaces in the level of generality discussed in \cite{MR3312633}.

\begin{definition} [Outer measure]
	Let $X$ be a set. An outer measure on $X$ is a function $\mu$ from the collection of all subsets of $X$ to $[0, \infty]$ that satisfies the following properties.
	\begin{enumerate}[(1)]
		\item $\mu(\emptyset) = 0$.
		\item If $E \subseteq F$ for two subsets of $X$, then $\mu(E) \leq \mu(F)$.
		\item If $\{E_i\}$ is a countable collection of sets in $X$, then
		\begin{equation*}
			\mu( \bigcup_{i=1}^\infty E_i) \leq \sum_{i=1}^\infty \mu(E_i).
		\end{equation*}
	\end{enumerate}
\end{definition}

\begin{definition} [Size]
	Let $X$ be a metric space. A size on $X$ is a function $S$ from the Cartesian product of $\mathcal{B}(X)$, the set of Borel measurable functions on $X$, and $\mathcal{P}(X)$, the power set of $X$, that satisfies, for every $f,g \in \mathcal{B}(X), A \subseteq X$, the following properties.
	\begin{enumerate}[(1)]
		\item If $\abs{f} \leq \abs{g}$, then $S(f)(A) \leq S(g)(A)$.
		\item $S(\lambda f)(A) = \abs{\lambda} S(f)(A)$ for every $\lambda \in \C$.
		\item There exists a constant $C$ depending only on $S$ but not on $f,g,A$ such that
		\begin{equation*}
			S(f+g)(A) \leq C [S(f)(A) + S(g)(A)].
		\end{equation*}
	\end{enumerate}
\end{definition}

We define
\begin{equation*}
	\norm{f}_{L^\infty(S)} = \sup_{A \subseteq X} S(f)(A),
\end{equation*}
and the outer $L^\infty(S)$ space to be the set of functions in $\mathcal{B}(X)$ for which this quantity is finite.

For $\lambda>0$, we define the super level measure
\begin{equation*}
	\mu(S(f)> \lambda) = \inf \{ \mu(A) \colon A {\textrm{ Borel subset of }} X, \norm{f1_{A^c}}_{L^\infty(S)} \leq \lambda \}.
\end{equation*}

For $0 < p < \infty$, we define
\begin{align*}
	\norm{f}_{L^{p}(S)} &= (\int_0^\infty p \lambda^p \mu(S(f) > \lambda) \frac{\diff \lambda}{\lambda})^{\frac{1}{p}}, \\
	\norm{f}_{L^{p,\infty}(S)} &= (\sup_{\lambda>0} \lambda^p \mu(S(f) > \lambda))^{\frac{1}{p}},
\end{align*}
and the outer $L^p(S), L^{p, \infty}(S)$ spaces to be the sets of functions in $\mathcal{B}(X)$ for which these quantities are finite, respectively. 

Finally, we recall some important results that hold in this setting.
\begin{proposition} [Pull back, Proposition 3.2 in \cite{MR3312633}] \label{thm:equivalence_outer_spaces}
	For $i=1,2$, let $X_i$ be a metric space, $\mu_i$ and $S_i$ be an outer measure and a size on $X_i$, respectively. Let $\Phi \colon X_1 \to X_2$ be a continuous map. Assume that for every $E_2 \in \mathcal{P}(X_2)$ we have
	\begin{equation*}
		\mu_1(\Phi^{-1}E_2) \leq A \mu_2(E_2).
	\end{equation*}
	Further assume that for each $E_1 \in \mathcal{P}(X_1)$, there exists $E_2 \in \mathcal{P}(X_2)$ such that for every $f \in \mathcal{B}(X_2)$ we have
	\begin{equation*}
		S_1(f \circ \Phi) (E_1) \leq B S_2(f) (E_2).
	\end{equation*}
	Then we have for every $f \in \mathcal{B}(X_2)$ and $0 < p \leq \infty$ and some universal constant $C$
	\begin{align*}
		\norm{f \circ \Phi}_{L^p(S_1)} &\leq A^{1/p} BC \norm{f}_{L^p(S_2)}, \\
		\norm{f \circ \Phi}_{L^{p,\infty}(S_1)} &\leq A^{1/p} BC \norm{f}_{L^{p,\infty}(S_2)}.
	\end{align*}
\end{proposition}
\begin{proposition} [outer H\"{o}lder's inequality, Proposition 3.4 in \cite{MR3312633}] \label{thm:Holder_inequality}
	Let $X$ be a metric space, $\mu,\mu_1,\mu_2$ be three outer measures on $X$ such that $\mu \leq \mu_i$, for $i=1,2$. Let $S,S_1,S_2$ be three respective sizes such that for any Borel subset $A$, there exist $A_1,A_2$ such that for all $f_1,f_2 \in \mathcal{B}(X)$ we have
	\begin{equation*}
		S(f_1 f_2)(A) \leq S_1(f_1)(A_1) S_2(f_2)(A_2).
	\end{equation*}
	Let $p,p_1,p_2 \in (0,\infty]$ such that $1/p = 1/p_1 + 1/p_2$. Then, for every $f_1,f_2 \in \mathcal{B}(X)$,
	\begin{equation*}
		\norm{f_1 f_2}_{L^p(S)} \leq 2 \norm{f_1}_{L^{p_1}(S_1)} \norm{f_2}_{L^{p_2}(S_2)}.
	\end{equation*}
\end{proposition}
\begin{proposition} [Marcinkiewicz interpolation] \label{thm:Marcinkiewicz_interpolation}
	Let $X$ be a measure space, $\mu$ and $S$ an outer measure and a size on $X$, respectively. Let $(Y,\nu)$ be a measure space. Let $1 \leq p_1 < p_2 \leq \infty, 1 \leq q_1 \neq q_2 \leq \infty$ such that $p_i \leq q_i$, for $i=1,2$. Let $T$ be a homogeneous quasi-subadditive operator that maps $L^{p_1}(Y,\nu)$ and $L^{p_2}(Y,\nu)$ to the space $\mathcal{B}(X)$ such that
	\begin{align*}
		\norm{T(f)}_{L^{q_1,\infty}(S)} \leq A_1 \norm{f}_{L^{p_1}(Y,\nu)}, \\
		\norm{T(f)}_{L^{q_2,\infty}(S)} \leq A_2 \norm{f}_{L^{p_2}(Y,\nu)}.
	\end{align*}
	Then we also have
	\begin{equation*}
		\norm{T(f)}_{L^q(S)} \leq A_1^\theta A_2^{1-\theta} C_{p,p_1,p_2} \norm{f}_{L^p(Y,\nu)},
	\end{equation*}
	where $0< \theta <1$ is such that
	\begin{equation*}
		\frac{1}{p}=\frac{\theta}{p_1}+\frac{1-\theta}{p_2}, \ \ \ \
		\frac{1}{p}=\frac{\theta}{q_1}+\frac{1-\theta}{q_2}.
	\end{equation*}
\end{proposition}
\begin{proof}
	See Appendix B in \cite{MR0290095}. It is enough, for a function $h$ on $X$, to replace the quantity $\mu(\{h > \lambda\})$ with the super level measure at level $\lambda$ in the definition of the non-increasing rearrangement $h^\ast$. In particular, for a function $h \colon X \to \R$, the function $h^\ast \colon (0,\infty) \to (0,\infty)$ is defined by
	\begin{equation*}
		h^\ast(t) = \inf \{ \lambda \colon \mu(S(h)> \lambda) \leq t \}.
	\end{equation*} 
\end{proof}
\begin{proposition} [Radon-Nikodym measure differentiation, Proposition 3.6 in \cite{MR3312633}, Proposition 1.9 in \cite{Uraltsev}] \label{thm:ell^1_domination}
	Let $X$ be a measure space, $\mu$ and $S$ an outer measure and a size on $X$, respectively. Assume that $X$ is $\sigma$-finite with respect to $\mu$, i.e.
	\begin{equation*}
		X = \bigcup_{n \in \N} X_n, \ \ \ \ \mu(X_n) < \infty.
	\end{equation*}
Let $\nu$ be a nonnegative Borel measure on $X$. Then, if either for all $A \subseteq X$
	\begin{equation*}
		\mu(A) = 0 \Rightarrow \nu(A) = 0,
	\end{equation*}
	or for all $A \subseteq X$
	\begin{equation*}
		\frac{1}{\mu(A)} \int_A \abs{f(x)} \diff \nu(x) \leq C \norm{f}_{L^\infty(S)},
	\end{equation*}
	we have, for every $f \in L^\infty(S)$,
	\begin{equation}
		\abs{\int_X \abs{f(x)} \diff \nu(x)} \leq C \norm{f}_{L^1(S)},
	\end{equation}
	where the implicit constant $C$ is independent of $\norm{f}_{L^\infty(S)}$.
\end{proposition}
\begin{proposition} \label{thm:log_convexity}
	Let $X$ be a measure space, $\mu$ an outer measure on $X$. For $0 < r_1 < r_2 \leq \infty$, let $\ell^{r_1},\ell^{r_2}$ be the sizes on $X$ defined by \eqref{eq:size}. Then, for every $0<p\leq \infty, r_1 < r < r_2$, there exists a constant $C=C(p,r,r_1,r_2)$ such that, for every $f \in \mathcal{B}(X)$,
	\begin{align*}
		\norm{f}_{L^p(\ell^r)} &\leq C (\norm{f}_{L^p(\ell^{r_1})}  + \norm{f}_{L^p(\ell^{r_2})}) , \\
		\norm{f}_{L^{p,\infty}(\ell^r)} &\leq C (\norm{f}_{L^{p,\infty}(\ell^{r_1})} + \norm{f}_{L^{p,\infty}(\ell^{r_2})}) .
	\end{align*}
\end{proposition}
\begin{proof}
	It is enough to prove that there exists a constant $c=c(r,r_1,r_2)$ such that, for every $\lambda >0$,
	\begin{equation*}
		\mu(\ell^r(f) > c \lambda) \leq C ( \mu(\ell^{r_1} (f) > \lambda) + \mu(\ell^{r_2} (f) > \lambda)).
	\end{equation*}
	The desired inequalities follow by multiplying the last display by $\lambda^p$ and either integrating or taking the supremum over all levels $\lambda >0$ . 
	
	Let $E_1, E_2 \subseteq X$ be two sets witnessing the super level measure at $\lambda$ up to a factor $2$ with respect to the sizes $\ell^{r_1}$ and $\ell^{r_2}$, respectively. In particular, we have
	\begin{equation*}
		2 \mu(\ell^{r_1}(f)>\lambda) \geq  \mu(E_1), \ \ \ \
		2 \mu(\ell^{r_2}(f)>\lambda) \geq  \mu(E_2).
	\end{equation*}
	Now let $E' = E_1 \cup E_2$. Then, for every $A \subseteq X$, we have
	\begin{equation*}
		\ell^r(f 1_{(E')^c}) (A) \leq c (\ell^{r_1} (f1_{(E_1)^c}) (A))^\theta (\ell^{r_1} (f1_{(E_2)^c}) (A))^{1-\theta} \leq c \lambda,
	\end{equation*}
	by logarithmic convexity of the $L^r$ spaces, where $0< \theta < 1$ satisfies
	\begin{equation*}
		\frac{1}{r} = \frac{\theta}{r_1} + \frac{1-\theta}{r_2}.
	\end{equation*}
To conclude, we observe that $\mu(E') \leq \mu(E_1) + \mu(E_2)$.
\end{proof}

\bibliographystyle{amsplain}
\bibliography{mybibliography}

\end{document}